\documentclass[12pt]{article}
\pdfoutput=1
\usepackage{latexsym, graphicx,cite}
\usepackage{amsmath}
\usepackage{amssymb}
\usepackage{amsthm}
\usepackage{mathtools}
\usepackage{graphicx}
\usepackage{caption}
\usepackage{subcaption}

\usepackage[top = 1. in,bottom = 1. in, left = 1 in, right=1 in]{geometry}

\newcommand{\arXiv}[1]{\href{http://www.arXiv.org/abs/#1}{arXiv:#1}}
\usepackage[colorlinks=true, linkcolor=blue, bookmarks=true]{hyperref}

\makeatletter
\renewcommand\section{\@startsection {section}{1}{\z@}%
	{-3.5ex \@plus -1ex \@minus -.2ex}%nn
	{2.3ex \@plus.2ex}%
	{\normalfont\large\bfseries}}
\renewcommand\subsection{\@startsection{subsection}{2}{\z@}%
	{-3.25ex\@plus -1ex \@minus -.2ex}%
	{1.5ex \@plus .2ex}%
	{\normalfont\bfseries}}
\makeatother

\newcommand{\beq}{\begin{equation}}
\newcommand{\eeq}{\end{equation}}
\newcommand{\beqnn}{\begin{equation*}}
\newcommand{\eeqnn}{\end{equation*}}
\newcommand{\ber}{\begin{array}}
	\newcommand{\eer}{\end{array}}
\newcommand{\del}{\partial}

\newcommand{\ssty}{\scriptstyle}

\newcommand{\te}{\theta}

\newcommand{\de}{\delta}

\newcommand{\ena}{\end{eqnarray}}
\newcommand{\beqa}{\begin{eqnarray}}
\newcommand{\eeqa}{\end{eqnarray}}
\newcommand{\bea}{\begin{eqnarray}}
\newcommand{\eea}{\end{eqnarray}}

\newcommand{\al}{\alpha}
\newcommand{\alb}{\bar{\alpha}}

\newcommand{\Ham}{\mathcal{H}}
\newcommand{\Sz}{Szeg\H{o}}
\newcommand{\cuSz}{cubic \Sz{} equation}

\theoremstyle{proposition}
\newtheorem{proposition}{Proposition}[section]
\theoremstyle{remark}

\begin{document}
	\begin{titlepage}
		\begin{flushright}
			\phantom{arXiv:yymm.nnnn}
		\end{flushright}
		\vspace{1cm}
		\begin{center}
			{\Large\bf Turbulent cascades in a truncation of the\vspace{3mm}\\ cubic \Sz{}  equation and related systems}\\
			\vskip 15mm
			{\large Anxo Biasi$^{1}$ and Oleg Evnin$^{2,3}$}
			\vskip 7mm
			{\em $^1$  Institute of Theoretical Physics, Jagiellonian University, Krak\'ow, Poland}
			\vskip 3mm
			{\em $^2$ Department of Physics, Faculty of Science, Chulalongkorn University,
				Bangkok, Thailand}
			\vskip 3mm
			{\em $^3$ Theoretische Natuurkunde, Vrije Universiteit Brussel and\\
				The International Solvay Institutes, Brussels, Belgium}
			\vskip 7mm
			{\small\noindent {\tt anxo.biasi@gmail.com, oleg.evnin@gmail.com}}
			\vskip 10mm
		\end{center}
		\vspace{1cm}
		\begin{center}
			{\bf ABSTRACT}\vspace{3mm}
		\end{center}
		
		\noindent We introduce a truncated version of the cubic \Sz{} equation, an integrable model for deterministic turbulence. In this truncation, a majority of the Fourier mode couplings are eliminated while the signature features of the model are preserved, namely, a Lax pair structure and a hierarchy of finite-dimensional dynamically invariant manifolds. Despite the impoverished structure of the interactions, the turbulent behaviors of our new equation are stronger in an appropriate sense than for the original cubic \Sz{} equation. We construct explicit analytic solutions displaying exponential growth of Sobolev norms. We furthermore introduce a family of models that interpolate between our truncated system and the original cubic \Sz{} equation, along with other related deformations. These models possess Lax pairs, invariant manifolds, and display a variety of turbulent cascades. We additionally mention numerical evidence, in some related systems, for an even stronger type of turbulence in the form of a finite-time blow-up.
		
		\vfill
		
		%\begin{flushleft}
		%PACS 11.25.-w, 04.65.+e
		%\end{flushleft}
	\end{titlepage}
	
	%%%%%%%%%%%%%%%%%%%%%%%%%%%%%%%%%%%%%%%%%%%%%%%%%%%%%%%%%%%%%%%%%%%%%%%%%%%%%%%%%%%%%%%%%%%%
	%%%%%%%%%%%%%%%%%%%%%%%%%%%%%%%%%%%%%%%%%%%%%%%%%%%%%%%%%%%%%%%%%%%%%%%%%%%%%%%%%%%%%%%%%%%%
	%%%%%%%%%%%%%%%%%%%%%%%%%%%%%%%%%%%%%%%%%%%%%%%%%%%%%%%%%%%%%%%%%%%%%%%%%%%%%%%%%%%%%%%%%%%%

\begin{flushright}
{\it \hspace{1cm}\\
 ``...the vessel underwent trials on the Paddington Canal in 1837.\\
However, by one of those fortunate accidents, which sometimes\\ 
occur in the history of science and technology, the propeller\\ 
was damaged during the trials and about half of it broke off,\\
whereupon the vessel immediately increased its speed.''} \cite{Carlton}\vspace{1.2cm}
\end{flushright}
	
	\section{Introduction\vspace{3mm}}\label{sec:Intro}
	
	In studies of Hamiltonian PDEs, the phenomenon of turbulence continues to challenge various communities of mathematicians and physicists. It is defined by transfer of energy from long-wavelength to short-wavelength modes, leading to concentration of energy on arbitrarily small spatial scales. Such {\em turbulent cascades} are usually quantified by the growth of Sobolev norms, and the general question of unbounded increase of Sobolev norms is a long-standing problem \cite{Bourgain}. Historically, much of the effort in this direction has focused on the wave turbulence theory \cite{Nazarenko} which employes averaging over the phases of weakly interacting waves. By contrast, studies of fully deterministic turbulence tend to be more recent, with much of the related efforts focused on the {\em nonlinear Schr\"{o}dinger equation} \cite{Hani1,Hani2,Colliander,Kuksin,Merle, Guardia, Guardia2, Guardia3, Guardia4}. Close parallels exist between this line of research and studies of weak turbulence in Anti-de Sitter spacetimes \cite{BR,rev2}, a topic of interest within general relativity and theoretical high-energy physics. 
	
	Motivated by studies of nonlinear Schr\"{o}dinger equations, G\'erard and Grellier designed a tractable model of non-dispersive evolution, the {\em cubic \Sz{} equation} \cite{GG}; see \cite{GG2} for a review. This nonlinear equation on the circle $\mathbb{S}^1$ is Lax-integrable and displays
turbulent behaviors that can be analyzed using integrability. To be more precise, it has been proved that there exist initial conditions with super-polynomial growth of some Sobolev norms; however, there are no explicit examples of such data. Existence of initial conditions with exponential growth of Sobolev norms is another open problem. Explicit solutions can be constructed that show `weak weak turbulence' in the language of \cite{Colliander}, so that the Sobolev norm growth is always bounded but can be arbitrarily enhanced by fine-tuning the initial data. These properties of the cubic \Sz{} equation motivated various studies of modifications of this model in order to elucidate whether such features are preserved and/or some additional phenomena emerge due to modifications. Examples that preserve integrability but display initial configurations with unbounded Sobolev norm growth in the past literature are \cite{Pocovnicu,Xu,Thirouin,GG3}. In \cite{Pocovnicu}, the cubic \Sz{} equation was placed on the real line $\mathbb{R}$, providing solutions with polynomial growth of Sobolev norms. In \cite{Xu}, the so-called {\em $\alpha$-\Sz{} equation} was introduced, with explicit examples of initial configurations with exponentially growing Sobolev norms. A few years later, the quadratic \Sz{} equation was introduced in \cite{Thirouin}. This case also displayed solutions with exponentially growing Sobolev norms. Finally, in \cite{GG3}, we can find the damped \Sz{} equation, likewise with unbounded Sobolev norm growth. 
\newpage	
The cubic \Sz{} equation and the other systems we discuss here belong to the class of cubic resonant systems. These resonant systems often arise as weakly nonlinear approximations to the dynamics of PDEs whose linearized spectra of normal frequencies are strongly resonant.
Thus, it was shown that the cubic \Sz{} equation accurately describes the weakly nonlinear long-time dynamics of some sectors of the half-wave equation \cite{GG,GG2} and the wave guide Schr\"odinger equation \cite{Xu2}. In closer contact with physics applications, resonant systems arise as approximations for the dynamics of Bose-Einstein condensates \cite{GHT,GT,BMP,BBCE,GGT} and in Anti-de Sitter (AdS) spacetimes \cite{FPU,CEV,BMR,CF,BEL,BCE,BEF}, the latter topic extensively studied in relation to AdS instability \cite{BR,rev2}. Resonant systems constructed in this way often show powerful analytic structures and admit special solutions \cite{BBE, breathing}, even in the absence of Lax-integrability that characterizes the cubic \Sz{} equation and other systems that we focus on here. 

As we have mentioned, considerable attention has been given to modifying the cubic \Sz{} equation in a way that strengthens  its turbulent behaviors or makes them more manifest. The main goal of our present exposition is to report our discovery of the {\em truncated \Sz{} equation} that possesses such features, together with systems interpolating between this new equation and the original cubic \Sz{} equation, and some further related deformations.
Relatively to what has been proposed in the literature before, our systems are close in spirit to those of \cite{Xu} and \cite{GG3}, though the concrete alteration we make in the equations is completely different. One might legitimately ask what is the purpose of introducing yet another modification of the cubic \Sz{} equation with such properties. First, for any system as special as the \cuSz, it is important to understand the full range of modifications that preserve its crucial properties (Lax-integrability, invariant manifolds, turbulence), and our results provide an extra four-parameter space of dynamical systems to contribute to this picture. Second, our equation is qualitatively distinct, say, from the modifications introduced in \cite{Xu,GG3}, which are its closest analogs, and may teach us important lessons. The modifications of \cite{Xu, GG3} are designed by adding a single simple term to the \cuSz. By contrast, our truncated \Sz{} equation results from {\it removing} most of the terms from (the Fourier representation of) the \cuSz. As turbulent energy transfer relies on mode couplings, the fact that turbulence is strengthened after removing a majority of mode couplings is surprising and counter-intuitive, and makes us think of the curious naval engineering incident from our epigraph, though the route by which we arrived at our equation was more systematic relatively to the XIX century naval precedent. Apart from their direct relevance in the context of the \cuSz{} studies, our findings may invite general re-evaluation of which mode couplings play an essential role in formation of turbulent cascades.
	
	The article is organized as follows: in section \ref{sec:Szego}, we provide some relevant details on resonant equations in general, as well as the cubic \Sz{} and $\alpha$-\Sz{} equations, which form the background material for our actual studies. In section \ref{sec:TR_Szego}, we introduce and examine our first model, the truncated \Sz{} equation. The Lax pair structure, invariant manifolds and explicit solutions with unbounded Sobolev norms are constructed. In section \ref{sec:beta_Szego}, we do the same with our second model, the $\beta$-\Sz{} equation. We provide summary and commentary, and mention some extra tentative results on finite-time turbulent blow-up in section \ref{sec:Discussion}.
	
	%%%%%%%%%%%%%%%%%%%%%%%%%%%%%%%%%%%%%%%%%%%%%%%%%%%%%%%%%%%%%%%%%%%%%%%%%%%%%%%%%%%%%%%%%%%%
	%%%%%%%%%%%%%%%%%%%%%%%%%%%%%%%%%%%%%%%%%%%%%%%%%%%%%%%%%%%%%%%%%%%%%%%%%%%%%%%%%%%%%%%%%%%%
	%%%%%%%%%%%%%%%%%%%%%%%%%%%%%%%%%%%%%%%%%%%%%%%%%%%%%%%%%%%%%%%%%%%%%%%%%%%%%%%%%%%%%%%%%%%%
	
	\section{Brief review of resonant equations, the cubic \Sz{} equation and the $\alpha$-\Sz{} equation}\label{sec:Szego}

	\subsection{Resonant equations}

The \cuSz{} and the new equations we shall introduce in this paper are all of the following algebraic form:
	\beq
	i\frac{d\alpha_n}{dt}=\hspace{-5mm}\sum_{\begin{array}{c}\vspace{-6mm}\\ \ssty m,k,l=0\vspace{-2mm}\\ \ssty n+m=k+l\end{array}}^{\infty}\hspace{-5mm} C_{nmkl}\bar{\alpha}_m{\alpha}_k\alpha_{l}.
	\label{eq:flow_eq_linear_spectrum}\vspace{-2mm}
	\eeq 
	Here,  $\alpha_n$ with $n=0,1,2,\ldots$ are complex-valued functions of time, which are our dynamical variables, bars denote complex conjugation, and $C_{nmkl}$ are real numbers that can be called the {\em mode couplings} or the {\em interaction coefficients}. The interaction coefficients are invariant under the following interchanges of the indices: $n\hspace{-0.0cm}\leftrightarrow\hspace{-0.0cm}m$, $k\hspace{-0.0cm}\leftrightarrow\hspace{-0.0cm}l$, $(n,m) \hspace{-0.0cm}\leftrightarrow\hspace{-0.0cm}(k,l)$. Such systems are often called {\em resonant systems} or {\em resonant equations}. Note the {\em resonant condition} $n+m = k+l$ restricting the summations in (\ref{eq:flow_eq_linear_spectrum}). Different representatives of this large class of equations that we shall study are distinguished by different explicit choices of the interaction coefficients $C$.

Equations of the form (\ref{eq:flow_eq_linear_spectrum}) commonly arise via application of
time-averaging or multi-scales analysis \cite{murdock,KM} as weakly nonlinear approximations to PDEs whose linearized spectrum of normal frequencies is highly resonant, as happens to a number of PDEs in harmonic traps or Anti-de Sitter spacetimes. We shall not review such derivations here, as it is quite far apart from our main focus, and will simply refer the reader to \cite{rev2,GHT,BBCE,murdock, KM}. We note that, viewed from this perspective, the dynamical variables $\al_n$ originate as the complex amplitudes of the linearized normal modes of the PDE, which acquire slow evolutions under the effect of weak nonlinearities. A system of the form (\ref{eq:flow_eq_linear_spectrum}) then accurately approximates this slow evolution.

The resonant equation (\ref{eq:flow_eq_linear_spectrum}) has a canonical Hamiltonian structure with a Hamiltonian of the form
	\beq
	\mathcal{H}={\frac12} \hspace{-2mm}\sum_{\begin{array}{c}\vspace{-6.5mm}\\ \ssty n,m,k,l=0\vspace{-2mm}\\ \ssty n+m=k+l\end{array}}^{\infty}\hspace{-3mm} C_{nmkl}\bar{\alpha}_n\bar{\alpha}_m{\alpha}_k\alpha_{l}\vspace{-3mm}
	\label{eq:H_cq}
	\eeq
and the symplectic form $i\sum_n d\bar\alpha_n\wedge d\alpha_n$. In addition to the Hamiltonian, there are two generic conserved quantities, irrespectively of the form of the interaction coefficients $C$, which we list together with the associated symmetry transformations:
	\begin{align}
	N = \sum_{n=0}^{\infty}|\alpha_n|^2, & \qquad \alpha_n(t) \to e^{i \phi} \alpha_n(t) \label{eq:N_consered_quantity},\\
	E = \sum_{n=0}^{\infty}n|\alpha_n|^2, & \qquad \alpha_n(t) \to e^{i n \theta} \alpha_n(t).  \label{eq:E_consered_quantity}
	\end{align}
	Additionally, equation (\ref{eq:flow_eq_linear_spectrum}) enjoys the scaling symmetry
	\beq
	\alpha_{n}(t) \to \epsilon \alpha_n(\epsilon^2 t).
	\eeq
For specific choices of the interaction coefficients, the set of symmetries and conserved quantities may, of course, become much bigger. This is the case for the class of systems treated in \cite{BBE}, for the \cuSz, and for the new systems we shall introduce below.

In full generality, equation (\ref{eq:flow_eq_linear_spectrum}) admits an infinite number of dynamically invariant manifolds. Namely, one can choose two mutually prime integers $p$ and $q$, $p<q$, and set to zero all $\al_n$ except for those with $n=p\hspace{-1mm}\mod q$. If this restriction is implemented in the initial conditions,  (\ref{eq:flow_eq_linear_spectrum}) guarantees that the modes set to zero will never get excited, which defines an invariant manifold of the evolution. We shall use such $(p\hspace{-1mm}\mod q)$-restrictions in some of our arguments. Again, for special systems within the large class given by  (\ref{eq:flow_eq_linear_spectrum}) , there can be many more invariant manifolds. Thus, for the systems at the focus of our current study, a crucial role is played by an infinite hierarchy of finite-dimensional invariant manifolds that we shall explicitly review below.

	\subsection{The cubic \Sz{} equation} \label{subsec:Szego}
	 The cubic \Sz{} equation was intensively studied in \cite{GG} and is usually given in the form
	 \beq
	 i\partial_t u = \Pi\left(|u|^2u\right),
	 \label{eq:Szego}
	 \eeq
	 where $\Pi$, the so-called \Sz{} projector, acts as a filter of negative Fourier modes
	\beq
	\Pi\left(\sum_{n\in\mathbb{Z}} \alpha_{n} e^{in\theta}\right) = \sum_{n=0}^{\infty} \alpha_{n} e^{in\theta}.
	\label{eq:Szego_projector}
	\eeq
	This equation is placed on the circle $\mathbb{S}^1$ and studied for $u\in L_{+}^2\left(\mathbb{S}^1\right)$; namely, the space of functions $L^{2}\left(\mathbb{S}^1\right)$ where negative Fourier modes are zero, $\alpha_n = 0$ $\forall n < 0$. Then, as $\Pi: L^2\left(\mathbb{S}^1\right)\to L_{+}^2\left(\mathbb{S}^1\right)$, (\ref{eq:Szego}) guarantees that for initial conditions $u_0\in L_{+}^2\left(\mathbb{S}^1\right)$
	\beq
		u(t, e^{i\theta}) = \sum_{n=0}^{\infty} \alpha_n(t) e^{i n\theta}, \qquad \text{with }\ \sum_{n=0}^{\infty} |\alpha_n|^2 < \infty.
		\label{eq:u_variable}
	\eeq
	The space $L^{2}\left(\mathbb{S}^1\right)$ is endowed with the inner product
	\beq
		\left(u|v\right) := \int_{\mathbb{S}^1}u\bar{v}\frac{d\theta}{2\pi}.
	\eeq
	The cubic \Sz{} equation admits a Hamiltonian structure with the Hamiltonian
	\beq
	\Ham_{Sz}=\frac{1}{4}\int_{\mathbb{S}^1}|u|^4  \frac{d\te}{2\pi} \qquad \text{with } u\in L_+^2(\mathbb{S}^1).
	\eeq

Substituting the Fourier expansion (\ref{eq:u_variable}) into (\ref{eq:Szego}), we rewrite the \cuSz{} in an equivalent form
	\beq
	i\frac{d\alpha_n}{dt}=\hspace{-5mm}\sum_{\begin{array}{c}\vspace{-6mm}\\ \ssty m,k,l=0\vspace{-2mm}\\ \ssty n+m=k+l\end{array}}^{\infty}\hspace{-5mm} \bar{\alpha}_m{\alpha}_k\alpha_{l}.
	\label{eq:Szego_modes}\vspace{-2mm}
	\eeq
This representation manifestly matches the general algebraic structure of the resonant equation (\ref{eq:flow_eq_linear_spectrum}), with the interaction coefficients given by the very simple
expression $C_{nmkl}^{\left(\text{Sz}\right)} = 1$.
We thus note that there are two different representations here to deal with: the position space representation (\ref{eq:Szego}) in terms of the \Sz{} projector, and the Fourier space representation (\ref{eq:Szego_modes}). The position space representation is more familiar from the literature on the \cuSz{}, and the \Sz{} projector language is effective, for example, for analyzing Lax integrability. The Fourier space representation makes more explicit contact with the physics of resonant systems. We shall rely on both representations in our derivations to highlight various aspects of our considerations.\vspace{5mm}

	Some relevant properties of the cubic \Sz{} equation are \cite{GG}:
	\begin{itemize}
		\item This system possesses two Lax pairs; see (\ref{LaxSz}) below.
		\item Sobolev norms\footnote{In terms of the Fourier modes, Sobolev norms can be expressed in the form $$\|u\|_{H^{s}}^2 = \sum_{n=0}^{\infty}\left(1+n\right)^{2s}|\alpha_n|^2.$$} of solutions of (\ref{eq:Szego}) subject to initial conditions $u_0\hspace{-0.12cm} \in\hspace{-0.12cm} H_{+}^{s}\left(\mathbb{S}^1\right)$  with $s\hspace{-0.1cm}>\hspace{-0.1cm}1$ (where $H_{+}^{s}\left(\mathbb{S}^1\right)\hspace{-0.12cm}:=\hspace{-0.12cm} H^s\left(\mathbb{S}^1\right)\hspace{-0.1cm}\cap\hspace{-0.1cm} L_{+}^2\left(\mathbb{S}^1\right)$), cannot grow faster than exponentially,
		\beq
			\|u(t)\|_{H^{s}} \leq C_{s}\|u_0\|_{H^{s}} e^{C_s \|u_0\|^2_{H^{s}}|t|} \qquad \text{with } C_{s}>0.
		\eeq
		\item There exist complex invariant manifolds $\mathcal{V}(d)$ consisting of functions of the form
		\beq
			u(z) = \frac{A(z)}{B(z)},
			\label{eq:V_manifold}
		\eeq
		where, $A(z)$ and $B(z)$ are polynomials in $z$ with no common factors, $B(0) = 1$, $B(z)$ has no zeros in the closed unit disk and 
		\begin{itemize}
			\item If $d = 2D$ is even, the degree of $A(z)$ is at most $D\hspace{-0.1cm}-\hspace{-0.1cm}1$ and the degree of $B(z)$ is exactly $D$.
			\item If $d = 2D+1$ is odd, the degree of $A(z)$ is exactly $D$ and the degree of $B(z)$ is at most $D$.
		\end{itemize} 
	These manifolds are linked to the Lax-integrability structure, as we shall explain below.
	\item Sobolev norms for initial conditions $u_0 \in \mathcal{V}(d)$ remain bounded for all times,
		\beq
			\forall s > \frac{1}{2} \qquad \underset{t\in\mathbb{R}}{\sup} \|u(t)\|_{H^{s}} < \infty.
			\label{eq:Szego_bond_infty}
		\eeq
		\item There are families of initial conditions $u_0^{\epsilon} \in \mathcal{V}(d)$ for a given $d$, with $\sup_{\epsilon} \|u_0^{\epsilon}\|_{H^s}<\infty$ for all $s$, such that
		\beq
			\forall s > \frac{1}{2} \qquad \underset{\epsilon}{\sup}\  \underset{t\in\mathbb{R}}{\sup} \|u^{\epsilon}(t)\|_{H^{s}} = \infty;
			\label{eq:Szego_epsilon_singular}
		\eeq
		namely, despite (\ref{eq:Szego_bond_infty}), we can fine-tune the initial conditions inside the same $\mathcal{V}(d)$ to get Sobolev norms with $s\hspace{-0.06cm}>\hspace{-0.06cm}1/2$ that grow as much as we please. (This is sometimes known as `weak weak turbulence' \cite{Colliander}.)
		In particular, the three-dimensional invariant manifold $\mathcal{V}(3)$ contains such solutions and will be of special interest in this paper.
		\item There exist initial conditions $u_0$ in $C^{\infty}\left(\mathbb{S}^1\right)\cap L_{+}^{2}\left(\mathbb{S}^1\right)$ such that for all $s> 1/2$ 
		\begin{align}
			\underset{t\to\infty}{\lim\text{sup}}\frac{\|u(t)\|_{H^s}}{t^{k}}= \infty& \qquad \forall k\geq 1,\\
			\underset{t\to\infty}{\lim\text{inf}}\|u(t)\|_{H^{s}} < \infty. & \ 
		\end{align}
		Namely, there is an infinite sequence of exchanges of energy back and forth between low and high modes, and the flow of energy to high modes provides for a super-polynomial growth of Sobolev norms with $s> 1/2$. We remark that currently there are no explicit examples of these initial data.
		\item The existence of initial conditions $u_0$, such that Sobolev norms with $s>1/2$ display an exponential growth is an open problem.
	\end{itemize}	
	 
	Now, to define the Lax pairs for (\ref{eq:Szego}), consider the following operators:
	\beq
	H_uh=\Pi(u\bar h),\qquad T_bh=\Pi(bh),\qquad Sh=e^{i\te}h.
	\label{eq:H_T_S_operators}
	\eeq
	In components, their action is
	\beq
	(H_u h)_n=\sum_{m=0}^\infty \al_{n+m}\bar h_m,\qquad (T_b h)_n=\sum_{m=0}^\infty b_{n-m} h_m, \qquad (Sh)_n=h_{n-1}.
	\label{eq:Hu_def}
	\eeq
	Note that $S$ is simply a shift changing the sequence $\{\al_0,\al_1,\ldots\}$ into $\{0,\al_0,\al_1,\ldots\}$, while the corresponding conjugate $S^\dagger$ does the opposite shift from $\{\al_0,\al_1,\ldots\}$ into $\{\al_1,\al_2,\ldots\}$. In particular, $S^\dagger S=1$.
	
	The two Lax pairs given in Theorem 3 in the third reference of \cite{GG} are 
	\beq
	\frac{dH_u}{dt}=[B_u,H_u],\qquad \frac{dK_u}{dt}=[C_u,K_u],
	\label{LaxSz}
	\eeq
	with
	\beq
		K_u = S^{\dagger}H_u=H_uS=H_{S^\dagger u}, \qquad B_u= \frac{i}2 H_u^2-i T_{|u|^2}, \qquad C_u=\frac{i}2 K_u^2-i T_{|u|^2},
		\label{eq:C_B_operators}
	\eeq
	whenever the equations of motion for $u$ are satisfied. The existence of invariant manifolds (\ref{eq:V_manifold}) can be understood in terms of these Lax operators. Actually, they correspond to the choices of $u$ for which the sum of the ranks of the operators $H_u$ and $K_u$ is $d$.
	
%%%%%%%%%%%%%%%%%%%%%%%%%%%%%%%%

	\subsection{The $\alpha$-\Sz{} equation} \label{subsec:alpha_Szego}
	The $\alpha$-\Sz{} equation \cite{Xu} was constructed as a deformation of the \cuSz{} by a term proportional to the lowest mode $\alpha_0 = (u|1)$:
	\beq
	i\partial_t u = \Pi\left(|u|^2u\right) + \alpha\left(u|1\right),
	\label{eq:alpha_Szego_v1}
	\eeq
	 where $\alpha\in\mathbb{R}$ and $u,\ \Pi$ and the operators $K_u$ and $C_u$ (that will appear later in this section), have the same definition as for the cubic \Sz{} equation. For any $\alpha~\neq~0$, the continuous dependence on this parameter can be absorbed by the rescaling $\tilde{u}(t) = \sqrt{|\alpha|}u(|\alpha|t)$, leaving the $\alpha$-\Sz{} equation as
	\beq
	i\partial_t \tilde{u} = \Pi\left(|\tilde{u}|^2\tilde{u}\right) + \text{sgn}(\alpha)\left(\tilde{u}|1\right).
	\label{eq:alpha_Szego_v2}
	\eeq 

Note that this model cannot be literally expressed in terms of the resonant equation (\ref{eq:flow_eq_linear_spectrum}), but the necessary deviation from the algebraic structure of (\ref{eq:flow_eq_linear_spectrum}) is small and only appears in the equation for the lowest mode:
		\beq
		i\frac{d\alpha_0}{dt}\mp \text{sgn}(\alpha)\alpha_0 =\hspace{-5mm}\sum_{\begin{array}{c}\vspace{-6mm}\\ \ssty m,k,l=0\vspace{-2mm}\\ \ssty m=k+l\end{array}}^{\infty}\hspace{-5mm} \bar{\alpha}_m{\alpha}_k\alpha_{l},
		  \qquad \text{and} \qquad  i\frac{d\alpha_n}{dt}=\hspace{-5mm}\sum_{\begin{array}{c}\vspace{-6mm}\\ \ssty m,k,l=0\vspace{-2mm}\\ \ssty n+m=k+l\end{array}}^{\infty}\hspace{-5mm} \bar{\alpha}_m{\alpha}_k\alpha_{l} \qquad \text{for }n \geq 1. 
		  \label{eq:alpha_Szego_modes}
		\eeq 
		The properties of this system are slightly different from the case of the cubic \Sz{} model \cite{Xu}:
		\begin{itemize}
			\item The $\alpha$-\Sz{} model possesses one Lax pair
			\beq
				\frac{dK_{u}}{dt} = \left[C_u, K_u\right].
				\label{eq:Lax_pair_truncate_Szego}
			\eeq
			\item Sobolev norms of solutions of (\ref{eq:alpha_Szego_v2}) subject to initial conditions $u_0 \in H_{+}^{s}\left(\mathbb{S}^1\right)$ with $s>1$, cannot grow faster than exponentially,
			\beq
			\|u(t)\|_{H^{s}} \leq \|u_0\|_{H^{s}} e^{C |t|} \qquad \text{with } C>0.
			\eeq
			\item There exist complex invariant manifolds $\mathcal{L}(D)$ defined by the choices of $u$ such that the rank of $K_u$ is $D$. They consist of rational functions of the form
			\beq
				u(z) = \frac{A(z)}{B(z)},
				\label{eq:L_odd}
			\eeq	
			where $A(z)$ and $B(z)$ are polynomials in $z$ of degree at most $D$, with no common factors, $\text{deg}(A) = D$ or $\text{deg}(B) = D$ and $B(z)$ has no zeros in the closed unit disk.

			\item For $\alpha < 0$ and initial conditions $u_0 \in \mathcal{L}(D)$, Sobolev norms are bounded
			\beq
			\forall s \geq 0 \qquad  \|u(t)\|_{H^{s}} < C.
			\eeq
			\item For $\alpha > 0$ and some $u_0 \in \mathcal{L}(1)$, Sobolev norms with $s\hspace{-0.08cm} >\hspace{-0.08cm} \frac{1}{2}$ grow exponentially for large enough~time,
			\beq
			\forall s > \frac{1}{2}, \qquad \|u(t)\|_{H^{s}} \underset{t\to\infty}{\simeq} e^{c(2s-1)|t|}.
			\eeq
			Hence, the $\alpha$-\Sz{} equation has solutions with unbounded Sobolev norms. 
			\item The $\alpha$-\Sz{} equation contains the cubic \Sz{} equation. Even for $\alpha\neq 0$, by restricting the initial conditions to odd modes (setting even modes to 0), (\ref{eq:alpha_Szego_modes}) is reduced to (\ref{eq:Szego_modes}). Hence, the $\alpha$-\Sz{} equation has subsectors of initial conditions with the properties displayed in subsection~\ref{subsec:Szego}.
		\end{itemize}	
		
	%%%%%%%%%%%%%%%%%%%%%%%%%%%%%%%%%%%%%%%%%%%%%%%%%%%%%%%%%%%%%%%%%%%%%%%%%%%%%%%%%%%%%%%%%%%%
	%%%%%%%%%%%%%%%%%%%%%%%%%%%%%%%%%%%%%%%%%%%%%%%%%%%%%%%%%%%%%%%%%%%%%%%%%%%%%%%%%%%%%%%%%%%%
	%%%%%%%%%%%%%%%%%%%%%%%%%%%%%%%%%%%%%%%%%%%%%%%%%%%%%%%%%%%%%%%%%%%%%%%%%%%%%%%%%%%%%%%%%%%%
	
	\section{The truncated \Sz{} equation}\label{sec:TR_Szego}
	
With the above preliminaries, we proceed with the key point of our presentation, which is the introduction of the {\em truncated \Sz{} equation}. Starting from the \cuSz, one simply sets to zero a specific large set of the interaction coefficients (a majority of them, in fact), while leaving the remaining ones intact, according to the pattern
	\beq
	C_{nmkl}^{\left(tr\right)} = \begin{cases}
		1 & \text{if } n m k l = 0,\\
		0 & \text{if } n m k l \neq 0.
	\end{cases}
	\label{eq:C_BE}
	\eeq
The condition that the product $nmkl$ must vanish evidently implies that at least one of the mode numbers $n$, $m$, $k$ or $l$ must vanish in order for the corresponding coupling coefficient $C$ to be nonzero.
	We have labelled the interaction coefficients of this truncated \Sz{} system  by $C^{\left(tr\right)}$ for future reference. By truncation, we simply mean eliminating interactions between modes (it should not be confused with restricting the dynamics of a given system to one of its invariant manifolds). Given the expression for $C^{\left(tr\right)}$, the equations of motion (\ref{eq:flow_eq_linear_spectrum}) take the form
	\beq
	i\frac{d\alpha_0}{dt} = \sum_{m=0}^{\infty} \sum_{k=0}^{m}\bar{\alpha}_m \alpha_k \alpha_{m-k}, \qquad i\frac{d\alpha_{n}}{dt} = \bar{\alpha}_0 \sum_{k=0}^{n}\alpha_k \alpha_{n-k} + 2\alpha_0 \sum_{m=1}^{\infty} \bar{\alpha}_m \alpha_{n+m} \quad \text{for } n \geq 1.
	\label{eq:resonant_eq_BE}
	\eeq
	 One can rewrite (\ref{eq:resonant_eq_BE}) in position space, i.e., in terms of $u$ given by (\ref{eq:u_variable}):
	\beq
	i\partial_t u = \Pi(|u|^2u)-S\Pi(|S^\dagger u|^2 S^\dagger u),
	\label{eq:truncated_Szego_eq}
	\eeq
	where $\Pi$ is the \Sz{} projector (\ref{eq:Szego_projector}) and $S$ the shift operator defined in (\ref{eq:H_T_S_operators}). The corresponding Hamiltonian is
	\beq
	\Ham_{tr}=\frac{1}{4}\int_{\mathbb{S}^1}\left(|u|^4-|S^\dagger u|^4\right) \frac{d\te}{2\pi} = \frac{1}{4}\sum_{n=0}^{\infty}\sum_{m=0}^{\infty}\sum_{k=0}^{n+m}C_{nmk(n+m-k)}^{\left(tr\right)}\bar\alpha_n\bar\alpha_m\alpha_k\alpha_{n+m-k}.
	\label{eq:Hamiltonian_truncated_Szego}
	\eeq
This Hamiltonian can be understood as the cubic \Sz{} system minus a ``shifted" cubic \Sz{} system, with $\{\alpha_0,\alpha_1,...\}$ replaced with 
$\{\alpha_1,\alpha_2,...\}$. The same pattern can be noticed in (\ref{eq:truncated_Szego_eq}). This structure of the model will be important for deriving properties of the truncated \Sz{} equation from the results previously known for the \cuSz.

	We now list the main properties of the truncated \Sz{} equation, which are the central technical results of our paper, and which will be proved in the remainder of this section:
	\begin{itemize}
		\item The truncated \Sz{} equation possesses one Lax pair, defined through the operators
(\ref{eq:H_T_S_operators}-\ref{eq:C_B_operators}):
		\beq
		\frac{dK_{u}}{dt} = \left[C_u-B_{S^\dagger u}, K_u\right].
\label{Laxtr}
		\eeq
		(See appendix A for an explicit action of these operators in  the mode representation.) Note that the presence of two Lax pairs in the original \cuSz{} is essentially used in the construction of its general solution in the third reference of \cite{GG}. Such derivations do not immediately generalize to the $\al$-\Sz{} equation or our system.
		\item Sobolev norms of solutions of (\ref{eq:truncated_Szego_eq}) subject to initial conditions $u_0 \in H_{+}^{s}\left(\mathbb{S}^1\right)$ with $s>1$, cannot grow faster than exponentially, which excludes a finite-time blow-up:
		\beq
		\|u(t)\|_{H^{s}} \leq \|u_0\|_{H^{s}} e^{C |t|} \qquad \text{with } C>0.
		\label{eq:truncated_Szego_global_bound}
		\eeq
		\item There exist complex invariant manifolds $\mathcal{L}(D)$ defined in (\ref{eq:L_odd}).
		\item For some $u_0\hspace{-0.1cm}\in\mathcal{L}(1)$, Sobolev norms with $s\hspace{-0.08cm} >\hspace{-0.08cm} \frac{1}{2}$ grow exponentially at late times:
		\beq
		\forall s > \frac{1}{2}, \qquad \|u(t)\|_{H^{s}} \underset{t\to\infty}{\simeq} e^{(2s-1)c|t|}.
		\label{eq:truncated_Szego_exponential_growth_Sobolev}
		\eeq
	\end{itemize}
	Before proceeding with our proofs, we add some extra comments:
	\begin{itemize}
		\item Integrability is a very fragile property, and a priori, one expects any modifications in the
mode coupling pattern to upset it. The specific modification of the mode couplings used to define the truncated \Sz{} equation is very special in this regard, as integrability is preserved (although we lose one of the two Lax pairs of the \cuSz{} equation). Note that the first Lax operator $K_u$ is common to the original and the truncated \Sz{} equations. As a consequence, one can construct a Lax pair for an arbitrary linear combination of $C^{\left(tr\right)}$ and $C^{\left(\text{Sz}\right)}$. This is the key idea behind our $\beta$-\Sz{} system to be introduced in section \ref{sec:beta_Szego}.
		\item The exponential growth of Sobolev norms (\ref{eq:truncated_Szego_exponential_growth_Sobolev}) is very surprising given the apparently impoverished structure of (\ref{eq:C_BE}). Interactions only through high modes have been completely eliminated, so that all the interactions must involve the lowest mode. Paradoxically, this enhances
the turbulent phenomena in $\mathcal{L}(1)$.
		\item In view of the exponential growth of Sobolev norms (\ref{eq:truncated_Szego_exponential_growth_Sobolev}) the bound (\ref{eq:truncated_Szego_global_bound}) is optimal.
		\item In this work, we will focus our attention on the properties of solutions in $\mathcal{L}(1)$. Dynamics in general $\mathcal{L}(D)$ is an open problem.
		\item The properties displayed here for the truncated \Sz{} equation and the ones for the\\ $\alpha$-\Sz{} equation showed in section \ref{subsec:alpha_Szego} are similar; however, we remark that the mode coupling structures of systems (\ref{eq:alpha_Szego_modes}) and (\ref{eq:resonant_eq_BE}) are very different. We shall return to further comparisons between these systems in our concluding section.
	\end{itemize}  

	 \subsection{Lax pair, bounded Hamiltonian and $\mathcal{L}(D)$ invariant manifolds}
	 The local and global well-posedness of the truncated \Sz{} equation for initial conditions $u_0 \in H_{+}^{s}\left(\mathbb{S}^1\right)$ with $s> 1/2$, as well as the exponential upper bound (\ref{eq:truncated_Szego_global_bound}), come from the results obtained in \cite{GG} for the \Sz{} equation and in \cite{Xu} for the $\alpha$-\Sz{} equation, essentially because the right-hand sides of (\ref{eq:resonant_eq_BE}) consist of subsets of terms that would have appeared in the case of the \cuSz. Specifically, after integrating (\ref{eq:truncated_Szego_eq}) in time
	 \beq
	 	u(t) = u_0 - i \int_{0}^{t}\left(\Pi(|u(s)|^2u(s))-S\Pi(|S^\dagger u(s)|^2 S^\dagger u(s))\right)ds,
	 \eeq
	 we have to use the estimates (where $\|u\|_{W} := \sum |\alpha_n|$ is the Wiener norm)
	 \begin{align*}
	 	& \|\Pi\left(|u|^2u\right)\hspace{-0.05cm}-\hspace{-0.05cm}S\Pi\left(|S^\dagger u|^2 S^\dagger u\right)\hspace{-0.15cm}\|_{H^{s}}  \hspace{-0.05cm} \leq \hspace{-0.05cm} \|\Pi\left(|u|^2u\right)\hspace{-0.15cm}\|_{H^{s}}\hspace{-0.05cm}+\hspace{-0.05cm}\|S\Pi\left(|S^\dagger u|^2 S^\dagger u\right)\hspace{-0.15cm}\|_{H^{s}}  \hspace{-0.05cm},\\
		& \|\Pi\left(|u|^2u\right)\hspace{-0.15cm}\|_{H^{s}} \leq c \|u\|_{L^\infty}^2 \|u\|_{H^s} \leq c \|u\|_{W}^2\|u\|_{H^s},\\
		& \|SS^\dagger u\|_{L^{\infty}} = \| u - (u|1)\|_{L^{\infty}} \leq \|u\|_{L^{\infty}}+\|(u|1)\|_{L^{\infty}} \leq 2\|u\|_{L^{\infty}},\\
		& \|S\Pi\left(|S^\dagger u|^2S^\dagger u\right)\hspace{-0.15cm}\|_{H^{s}} \hspace{-0.05cm} \leq \hspace{-0.05cm} \|\Pi\left(|SS^\dagger u|^2 SS^\dagger u\right)\hspace{-0.15cm}\|_{H^{s}} \hspace{-0.05cm} \leq \hspace{-0.05cm} \hat{c} \| u\|_{L^\infty}^2 \|u\|_{H^s} \hspace{-0.05cm} \leq \hspace{-0.02cm} \hat{c} \|u\|_{W}^2 \|u\|_{H^s},
		\end{align*}
		the Brezis-Gallou\"et type estimate\footnote{A proof of this estimate can be found in the appendix of the first reference of \cite{GG}.}
		$$ \|u\|_{L^{\infty}} \leq C \|u_0\|_{H^{\frac{1}{2}}} \left(\log \left(2 + \frac{\|u\|_{H^{s}}}{\|u_0\|_{H^{\frac{1}{2}}}}\right)\right)^{1/2},
		$$
		the upper bound for the Wiener norm for $s>1$ \cite{Xu},
		\beq \underset{t\in\mathbb{R}}{\text{sup}} \|u\|_{W} \leq \tilde{c}_s \|u_{0}\|_{H^s},		
		\eeq
	 and the Gronwall lemma. Note that the bound on the Wiener norm is proved in \cite{Xu} only relying on the fact that $K_u$ is a Lax operator. Since that holds true for our current model, as we shall immediately demonstrate, the proof of \cite{Xu} directly translates to our case.
	 
	 To establish the Lax pair (\ref{Laxtr}), we recall the ``\Sz{}-minus-shifted-\Sz{}" structure of the truncated \Sz{} system. Then, one can effectively reuse the Lax pairs (\ref{LaxSz}). 
	 \begin{proposition}
		Let $u\in C\left(\mathbb{R},H^{s}\left(\mathbb{S}^1\right)\right)$ for $s> 1/2$, if $u$ solves the truncated \Sz{} equation (\ref{eq:truncated_Szego_eq}), then $\left(K_u, D_u\right)$ with $D_u = C_u-B_{S^\dagger u}$ satisfy
		\beq
			\frac{dK_{u}}{dt} = \left[D_u, K_u\right],
\label{DK}
		\eeq
i.e., they provide a Lax pair.
	 \end{proposition}
	 \begin{proof}
	 	For this proof, we will use the Lax pairs (\ref{LaxSz}) for the cubic \Sz{} equation. We consider $dK_u/dt$ evaluated with the equation of motion (\ref{eq:truncated_Szego_eq}). Since $K_u$ is linear in $u$, we have
	 	\beq
	 	\frac{dK_u}{dt}=-i K_{\Pi(|u|^2u)}+iK_{S\Pi(|S^\dagger u|^2 S^\dagger u)}.
	 	\eeq
	 	The first term is simply what one would have gotten for the \cuSz{} itself, and hence it equals $[C_u,K_u]$ by (\ref{LaxSz}). The second term can be written as follows, taking into account (\ref{LaxSz}),
	 	\beq
	 	-iK_{S\Pi(|S^\dagger u|^2 S^\dagger u)}=-iH_{S^\dagger S\Pi(|S^\dagger u|^2 S^\dagger u)}= -iH_{\Pi(|S^\dagger u|^2 S^\dagger u)}=[B_{S^\dagger u}, H_{S^\dagger u}]=[B_{S^\dagger u}, K_{u}].
	 	\eeq
	 	Hence,
	 	\beq
	 	\frac{dK_u}{dt}=[C_u-B_{S^\dagger u},K_u],
	 	\eeq
	 	and a Lax pair for the truncated \Sz{} system is given by $K_u$ and
	 	\beq
		 	 D_u=C_u-B_{S^\dagger u}=iT_{|S^\dagger u|^2}-iT_{|u|^2}=-iT_{|u|^2-|S^\dagger u|^2}.
	 	 \eeq
(We give an alternative verification of the Lax pair in appendix A using the language of the mode space. This derivation is lengthier but more straightforward and self-contained.)
	 \end{proof}
	 
	 An important fact that we must clarify is whether the Hamiltonian (\ref{eq:Hamiltonian_truncated_Szego}) is bounded from below. While the Hamiltonian for the \cuSz{} is positive, removing interaction coefficients to obtain $C^{\left(tr\right)}$ could undermine the existence of a lower bound. This does not in
	 fact occur:
	 \begin{proposition}
	 	Given initial conditions $\{\alpha_n(0)\}$ such that the conserved quantities N and E defined in (\ref{eq:N_consered_quantity}-\ref{eq:E_consered_quantity}) are finite, the Hamiltonian (\ref{eq:Hamiltonian_truncated_Szego}) is bounded from below by
	 	\beq
		 	\Ham_{tr} \geq - N E.
		 	\label{eq:H_bounde_truncated_Szego}
	 	\eeq
	 \end{proposition}
	\begin{proof}
		The Hamiltonian (\ref{eq:Hamiltonian_truncated_Szego}) is rewritten in the following form
		\beq
		\Ham_{tr}=\frac{|\al_0|^4}{4}+|\al_0|^2\sum_{n=1}^\infty |\al_n|^2+ \frac{1}{2}\sum_{n=2}^\infty \sum_{k=1}^{n-1} \left(\alb_0 \alb_{n} \al_k \al_{n-k}+ \alb_k \alb_{n-k}\al_0 \al_{n}\right).
		\label{Htr}
		\eeq
		Note that
		$$
		\sum_{n=2}^\infty \sum_{k=1}^{n-1} \left(\alb_0 \alb_{n} \al_k \al_{n-k}+ \alb_k \alb_{n-k}\al_0 \al_{n}\right)=\sum_{n=2}^\infty\left[\Big|\al_0 \al_{n}+\sum_{k=1}^{n-1} \al_k \al_{n-k}\Big|^2-|\al_0\al_n|^2-\Big|\sum_{k=1}^{n-1} \al_k \al_{n-k}\Big|^2\right],
		$$
		and also that
		$$
		\sum_{n=2}^\infty |\al_0\al_n|^2 \le |\al_0|^2\sum_{n=1}^\infty |\al_n|^2.
		$$
		At the same time, by the Cauchy-Schwartz inequality,
		$$
		\Big|\sum_{k=1}^{n-1} \al_k \al_{n-k}\Big|^2\le \sum_{k=1}^{n-1} |\al_k|^2 |\al_{n-k}|^2 \sum_{l=1}^{n-1} 1=(n-1) \sum_{k=1}^{n-1} |\al_k|^2 |\al_{n-k}|^2= \sum_{k=1}^{n-1} (2k-1)|\al_k|^2 |\al_{n-k}|^2.
		$$
		Hence,
		\begin{align*}
		&\sum_{n=2}^\infty \Big|\sum_{k=1}^{n-1} \al_k \al_{n-k}\Big|^2\le \sum_{k=1}^{\infty} (2k-1)|\al_k|^2  \sum_{l=1}^{\infty}|\al_{l}|^2\\
&\hspace{3.5cm}=(2E-N+|\al_0|^2)(N-|\al_0|^2)=2NE-2E|\al_0|^2-(N-|\al_0|^2)^2.
		\end{align*}
		Combining everything together, we obtain the bound (\ref{eq:H_bounde_truncated_Szego}).
	\end{proof}

	As the existence of invariant manifolds $\mathcal{L}\left(D\right)$ is linked to the Lax operator $K_u$, they are predictably respected by the truncated \Sz{} evolution:
\begin{proposition}
	For every positive integer $D$ there exists a complex manifold $\mathcal{L}(D)$, given in (\ref{eq:L_odd}), that is invariant under the evolution of the truncated \Sz{} equation (\ref{eq:resonant_eq_BE}).
\end{proposition}
\begin{proof}
	Since our $K_u$ is exactly the same as for the cubic \Sz{} and $\alpha$-\Sz{} equations, we could simply verbatim recapitulate the analysis of \cite{Xu}.
Instead of reproducing this proof, we shall sketch here an elementary proof that remains valid for as long as no two simple poles of $u$ collide. We are going to show that the truncated \Sz{} equation reduces to $2D+1$ ODEs for $2D+1$ variables, making the restriction to such manifolds consistent with the evolution, when the initial conditions are restricted to $\mathcal{L}(D)$ and assuming that $u(t=0,z)$ only has simple poles. The Fourier coefficients of $u(z)$ are then of the form
	\beq
		\alpha_{n \geq 1}(t) = \sum_{k=1}^{D}c_k(t) p_k(t)^n, \qquad \hspace{-0.5cm} \alpha_0(t) = b(t) +  \sum_{k=1}^{D}c_k(t), \quad \hspace{-0.3cm} |p_k|<1 , \quad \hspace{-0.3cm} p_i\neq p_j.
		\label{eq:L_D_single_poles}
	\eeq

	We start our analysis with the equations for $\al_{n \geq 1}$, returning to the equation for $\al_0$ at the end. The left-hand side of (\ref{eq:resonant_eq_BE}) 
		reduces to $D$ terms of the form $i \left(\dot{c}_i + n c_i \frac{\dot{p}_i}{p_i}\right) p_i^n$. Now we are going to show that the right-hand side is decomposed into $D$ terms with the same structure; namely, $p_i^n$ times a linear function of $n$. After substituting (\ref{eq:L_D_single_poles}) on the right-hand side of (\ref{eq:resonant_eq_BE}), we can show that the second sum is decomposed into terms of the form $A_i(b,p,c)p_i^n$ where $A_i(b,p,c)$ do not depend on $n$. The first sum has terms of the form
		\beq
			B_{ij}(b,c,p)\sum_{k=0}^{n}p_i^{k}p_j^{n-k}.
		\eeq
		When $i = j$ it becomes $B_{jj}(b,p,c) (n+1) p_j^n$ and when $i\neq j$ (remember that $p_i\neq p_j$ and $|p_k| < 1$)
		\beq
		B_{ij}(b,c,p)\sum_{k=0}^{n}p_i^{k}p_j^{n-k} = B_{ij}(b,p,c) \frac{p_i^{n+1}-p_j^{n+1}}{p_i-p_j}.
		\eeq
		As a consequence, gathering all the terms with $p_i^{n}$ on the right-hand side, we get $D$ terms whose $n$-dependence is of the form $p_i^{n}$ times a linear function of $n$. Matching the left-hand side and the right-hand side, we obtain $2D$ equations for $2D+1$ variables. The remaining equation, the one for $b(t)$, comes from expressing $\dot b$ using the equation of motion for ${\alpha}_0$ and the $2D$ equations for $\dot p_i$ derived above.
		
A more complete (but less elementary) proof that remains valid even if poles collide can be given based on the properties of $K_u$ as in \cite{Xu}.
	\end{proof}
	%%%%%%%%%%%%%%%%%%%%%%%%%%%%%%%%%%%%%%%%%%%%%%
	
	\subsection{Explicit blow-up in $\mathcal{L}(1)$}\label{subsec:TR_Szego_3d-manifold}
	
Having described the structure of invariant manifolds, we shall now focus on the dynamics in $\mathcal{L}(1)$, and demonstrate strong explicit turbulent behaviors there, which is a proof of (\ref{eq:truncated_Szego_exponential_growth_Sobolev}). To this end, solutions in $\mathcal{L}(1)$ are parametrized as
	\beq
	\alpha_0(t) = b(t), \qquad \alpha_{n\geq 1}(t) = \left(b(t)p(t) + a(t)\right)p(t)^{n-1},
	\label{eq:ansatz_3d} 
	\eeq
	where $b,\ a$ and $p$ are complex functions of time. With this ansatz, the equations of motion in (\ref{eq:resonant_eq_BE}) are reduced to a coupled system of ODEs given by
	\begin{align}
	i \dot{p} = & \left(N - \left(1-|p|^2\right)E\right) p + a\bar{b}, \label{eq:truncated_szego_pdot}\\
	i \dot{b} = & (N+E)b + E a \bar{p}, \\
	i \dot{a} = & (N-E) a - E  |p|^2 b p, \label{eq:truncated_szego_adot}
	\end{align}
	where we have made use of the following expressions for the conserved quantities (\ref{eq:H_cq}-\ref{eq:E_consered_quantity})
	\beq
	E = \frac{|bp+a|^2}{\left(1-|p|^2\right)^2}, \qquad N = |b|^2 + \left(1 - |p|^2\right)E,
	\label{eq:E_J_3d_ansatz}
	\eeq
	\beq
	\Ham_{tr} = \frac{1}{4}\left(N^2 + 2 E N - 3 E^2 + 2E S\right), \qquad \text{with } S = (N+E)|p|^2 + \frac{E}{2} |p|^4 +  (a \bar b \bar p + \bar a b p).
	\label{eq:H_S}
	\eeq
	Note that the conservation of $N,\ E$ and $\Ham_{tr}$ implies the conservation of $S$. Through these expressions for the conserved quantities, the evolution of $|\alpha_n(t)|^2$ can be written in terms of $|p(t)|^2$ as
	\beq
	|\alpha_0(t)|^2 = N - \left(1 - |p(t)|^2\right)E, \qquad |\alpha_{n\geq 1}(t)|^2 = \left(1 - |p(t)|^2\right)^2 E |p(t)|^{2n-2}.
	\label{eq:alpha_n_as_N_E_x}
	\eeq
	 Therefore, our study will be focused on the equation for $|p(t)|^2$.
	After elementary algebra utilizing (\ref{eq:E_J_3d_ansatz}-\ref{eq:H_S}), and using the notation $x(t) := |p(t)|^2$, this equation can be expressed in the form
	\beq
	\dot{x}^2 + c_0 + c_1 x + c_2 x^2 + c_3 x^3 + c_4 x^4 = 0,
	\label{eq:xdot_eq_1}
	\eeq
	with the coefficients
	\beqnn
	c_4 =  -\frac{7}{4}E^2, \qquad c_3 =  3E^2 - N E, \qquad c_2 =  N^2 + 2 N E - 3 E^2 + 3E S,
	\eeqnn
	\beq
	c_1 =  4E^2 + 2 N S - 4NE - 6 E S, \qquad c_0 =  S^2.
        \label{Veffc1}
	\eeq
	This equation can be understood as the energy conservation for zero energy trajectories of a quartic nonlinear oscillator, which can be manifested by rewriting it as
	\beq
	\dot{x}^2 + V_{\text{eff}}(x) = 0,
	\label{eq:xdot_eq_2}
	\eeq
with $V_{\text{eff}}$ read off  (\ref{eq:xdot_eq_1}-\ref{Veffc1}). 
	Consequently, given an initial condition, through a standard analysis of $V_{\text{eff}}(x)$ we can determine whether $|\alpha_n(t)|^2$ is static, periodic or whether Sobolev norms with $s > 1/2$ are unbounded. See fig.~\ref{fig:shapes_Veff} for illustrative pictures of $V_{\text{eff}}(x)$ associated with each of these behaviors.
	
	\begin{figure}[t!]
		\centering
		\begin{subfigure}[b]{0.5\textwidth}
			\centering
			\includegraphics[width=8cm]{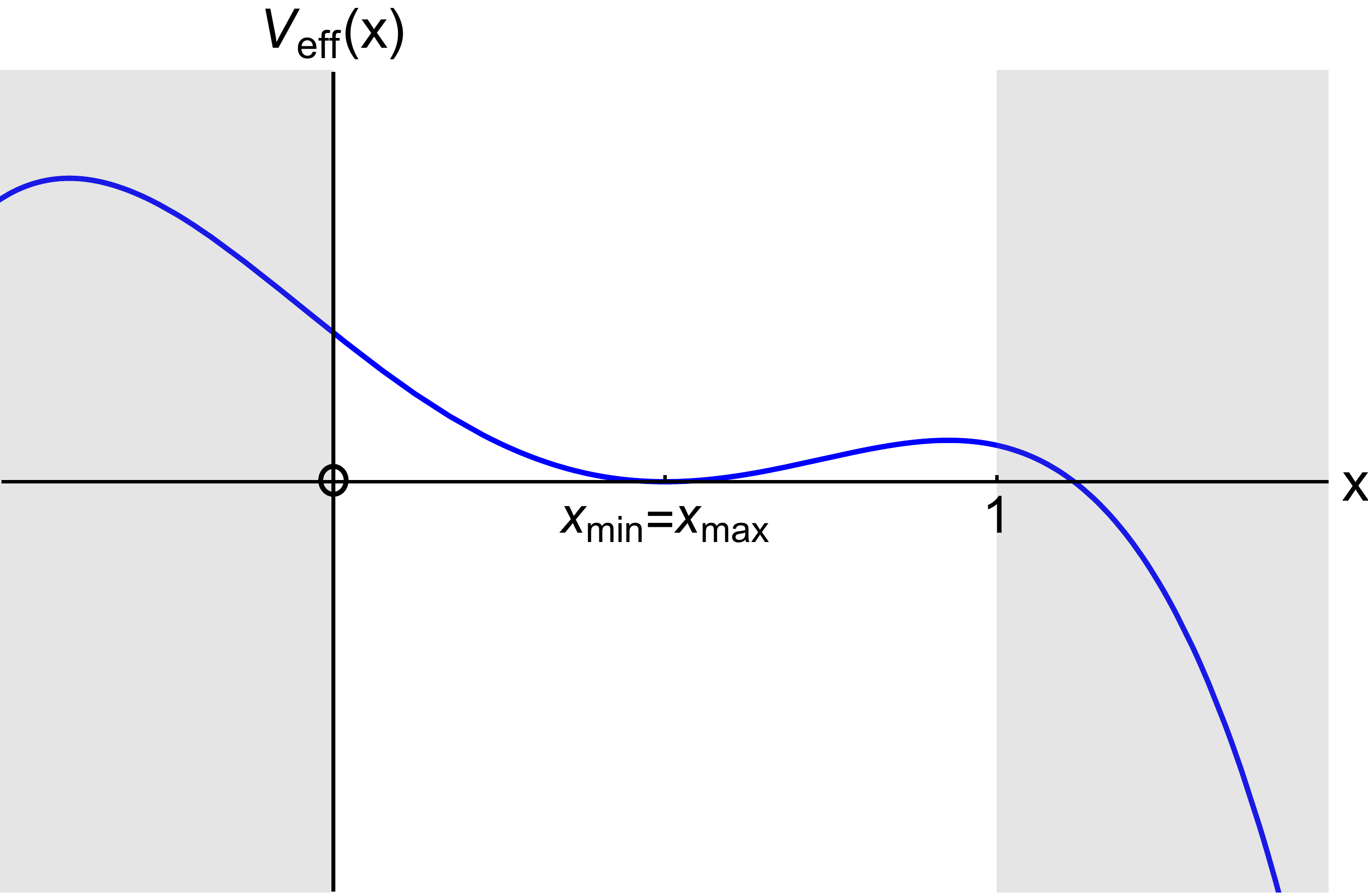}
			\caption{\small Stationary solution.}
			\label{fig:shapes_Veff_a}
		\end{subfigure}%
		\begin{subfigure}[b]{0.5\textwidth}
			\centering
			\includegraphics[width=8cm]{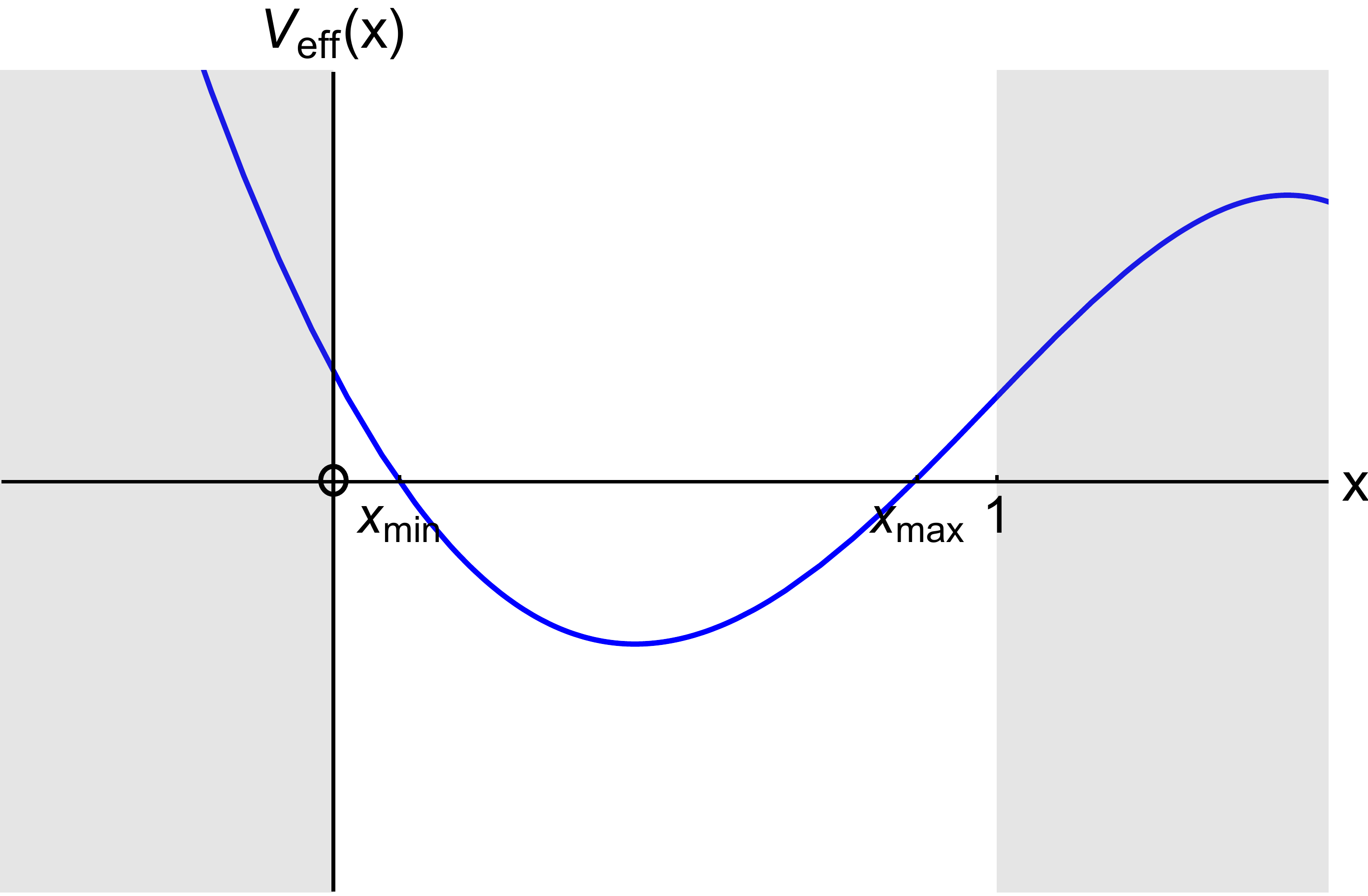}
			\caption{\small Periodic solution.}
			\label{fig:shapes_Veff_b}
		\end{subfigure}%
		
		\begin{subfigure}[b]{0.5\textwidth}
			\centering
			\includegraphics[width=8cm]{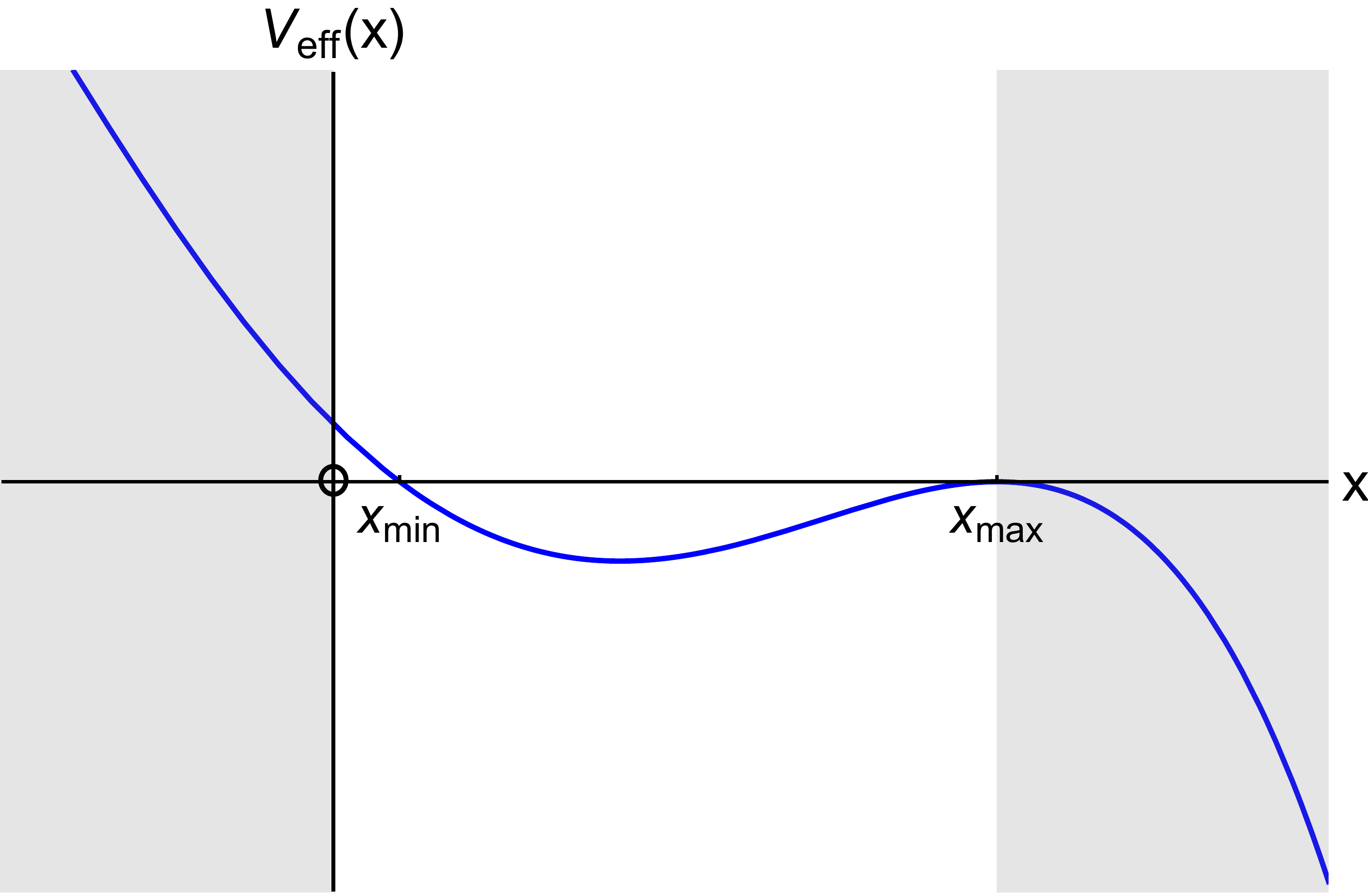}
			\caption{\small Infinite-time blow-up solution.}
			\label{fig:shapes_Veff_c}
		\end{subfigure}%
		\begin{subfigure}[b]{0.5\textwidth}
			\centering
			\includegraphics[width=8cm]{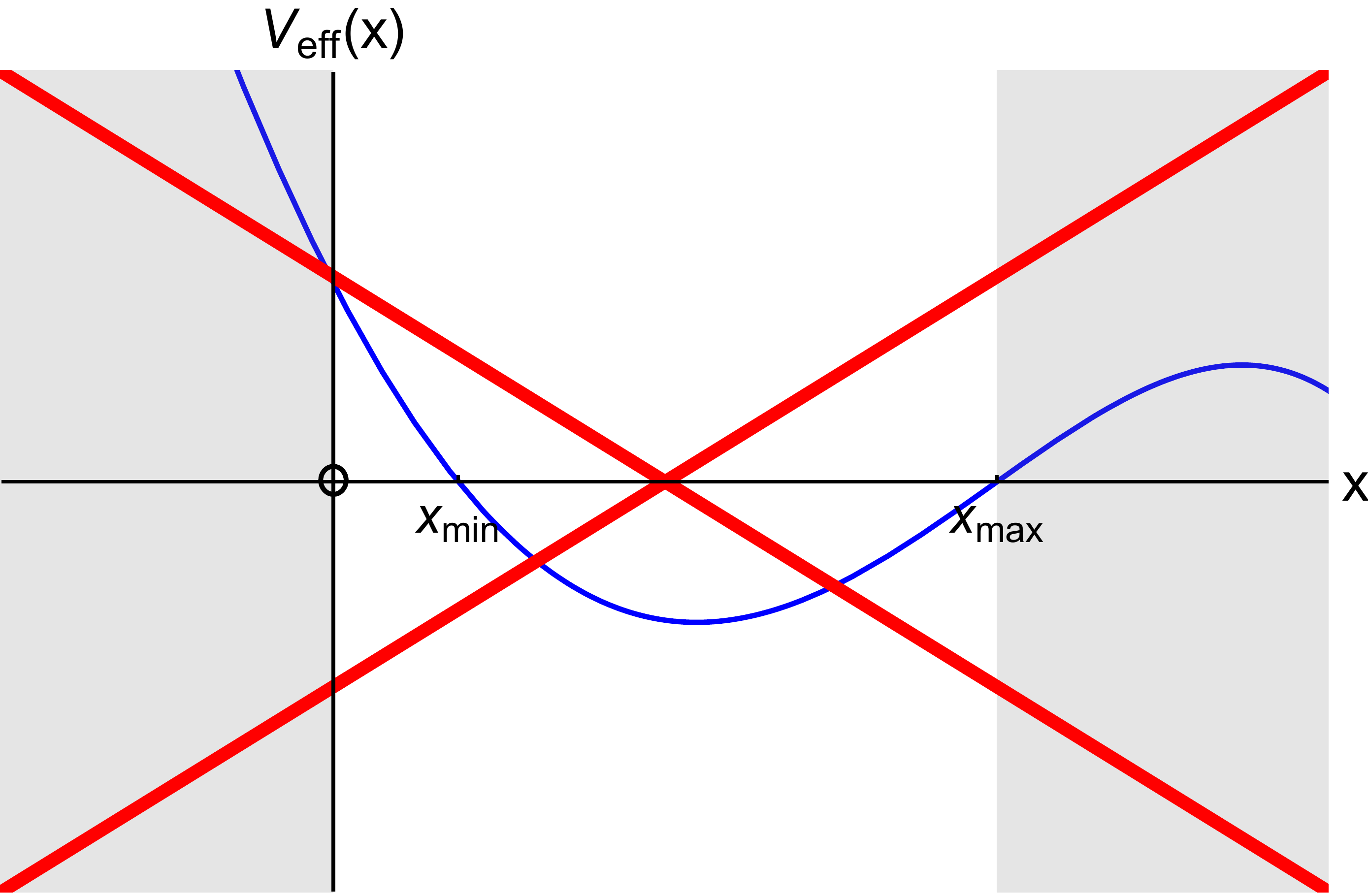}
			\caption{\small Finite-time blow-up solution.}
			\label{fig:shapes_Veff_d}
		\end{subfigure}%
		\caption{\small The shapes of the effective potential for different sets of initial conditions, corresponding to different dynamical regimes of our system. The shaded areas do not correspond to valid configurations of the dynamical variables (they were included to provide a more complete picture of $V_{\text{eff}}(x)$).  The center of the black circle denotes the origin $(0,0)$ and $x_{\text{min}}$ ($x_{\text{max}}$) represents the minimum (maximum) value reached by $x(t)$. Dynamics of the type $(a),\ (b)$ and $(c)$ actually occurs in the evolution of the truncated \Sz{} equation, while $(d)$ is impossible due to the upper bound (\ref{eq:truncated_Szego_global_bound}), and the plot is crossed out in red to highlight this fact. (One should view $(d)$ as an illustration of an imaginary shape of $V_{\text{eff}}(x)$ for which a finite-time blow-up would emerge.)}
		\label{fig:shapes_Veff}
	\end{figure}
	
	We shall now construct explicit initial configurations within the ansatz (\ref{eq:ansatz_3d}) whose evolution displays unbounded Sobolev norms. The exponential bound (\ref{eq:truncated_Szego_global_bound}) prevents any finite-time blow-up, a feature that is reflected in the potential as $V_{\text{eff}}(1)\geq 0$ and $V_{\text{eff}}(1)=0\Leftrightarrow V_{\text{eff}}'(1)~=~0$. The only kind of turbulence present is in the  form of unbounded Sobolev norms growth over an infinite range of time. The corresponding type of the effective potential can be seen in fig.~\ref{fig:shapes_Veff_c}. This set of configurations must satisfy $V_{\text{eff}}(1)=0$, a condition that, in terms of the conserved quantities, takes the form
	\beq
		3E-2(N+S) = 0
	\eeq
	and in terms of the dynamical variables in $\mathcal{L}(1)$,
	\beq
		2|b+a\bar{p}|^2 = |bp+a|^2(1+|p|^2).
	\eeq
	The solution is
	\beq
	a = b p \left(1 - e^{i\lambda} \sqrt{2 \left(1+ \frac{1}{|p|^2} \right)}\right) \qquad \text{with }  \lambda \in \mathbb{R},
	\label{eq:a0_general_ID}
	\eeq
	indicating that only initial conditions within a lower-dimensional submanifold in $\mathcal{L}(1)$ exhibit unbounded Sobolev norm growth. This situation is directly parallel to the $\al$-\Sz{} equation \cite{Xu}.
	Assuming (\ref{eq:a0_general_ID}), the effective potential of (\ref{eq:xdot_eq_2}) takes the form
	\beq
	V_{\text{eff}}(x) = - |b(0)|^4 F \left(1-x\right)^2 (x_{\text{min}} - x) (c - x),
	\eeq
	where $c,\ x_{\text{min}}$ and $F$ are functions of $|p(0)|^2$ and $\lambda$ satisfying the bounds $F>0$, $0\leq x_{\text{min}} < 1$ and $c < x_{min}$ for $0\leq|p(0)|^2 < 1$. Furthermore, $x_{\text{min}}$ denotes the minimum value of $x(t)$, and $c$, another real zero of $V_{\text{eff}}(x)$. The solution of equation (\ref{eq:xdot_eq_2}) subject to this structure of $V_{\text{eff}}(x)$ is
	\beq
	x(t) = \frac{\left(c - x_{\text{min}}\right)\cosh{\left(\omega t + \phi\right)} - c - x_{\text{min}} + 2 c x_{\text{min}}}{\left(c - x_{\text{min}}\right)\cosh{\left(\omega t + \phi\right)} +c + x_{\text{min}} - 2}, \quad \text{with} \quad \omega = |b(0)|^2\sqrt{F(1-c)(1-x_{\text{min}})}
	\eeq
	and $\phi$ such that $x(0) = |p(0)|^2$. The inequalities $c< x_{\text{min}}$, $0\leq x_{\text{min}} < 1$ and $F > 0$ guarantee that if $|b(0)|\neq 0$, then $\omega>0$ and also that $(c - x_{\text{min}})\neq 0$. Hence, $x(t)\equiv |p(t)|^2$ exponentially approaches $1$ at late times:
	\beq
	|p(t)|^2 \sim 1 - \mathcal{O}\left(e^{-\omega t}\right).
	\eeq
	Then, using the expressions for $|\alpha_n(t)|^2$ in terms of $N,\ E$ and $|p(t)|^2$ given in (\ref{eq:alpha_n_as_N_E_x}), Sobolev norms for $s > 1/2$ have the following exponential growth at late times
	\beq
	\|u(t)\|_{H^{s}} \underset{t\to\infty}{\simeq} e^{(2s-1)\frac{\omega}{2} |t|}.
	\label{eq:u_H_s_exp_growth_truncated_szego}
	\eeq
	In fig.~\ref{fig:collapse_wave_function}, we show the evolution of
	\beq
	v(t,\theta) = \sum_{n=1}^{\infty} \alpha_n(t) e^{in\theta}
	\label{eq:v_function}
	\eeq
for one of the initial conditions displaying such exponential growth of Sobolev norms. This function shows a concentration phenomenon, meaning that $v(t,\theta)$ tends to 0 at large $t$ everywhere, except for a single value of $\theta$, where it tends to $4E$. Evidently, such behavior must incur unbounded growth of the derivatives of $v$, which is in turn reflected in the growth of Sobolev norms. Note that $v(t,\theta)$ itself could not possibly blow up by conservation of $N$ and $E$. The blow-up necessarily enters through the derivatives of $v$.

	\begin{figure}[t!]
	\centering	
	\begin{subfigure}[b]{0.5\textwidth}
		\centering
		\includegraphics[width=8cm]{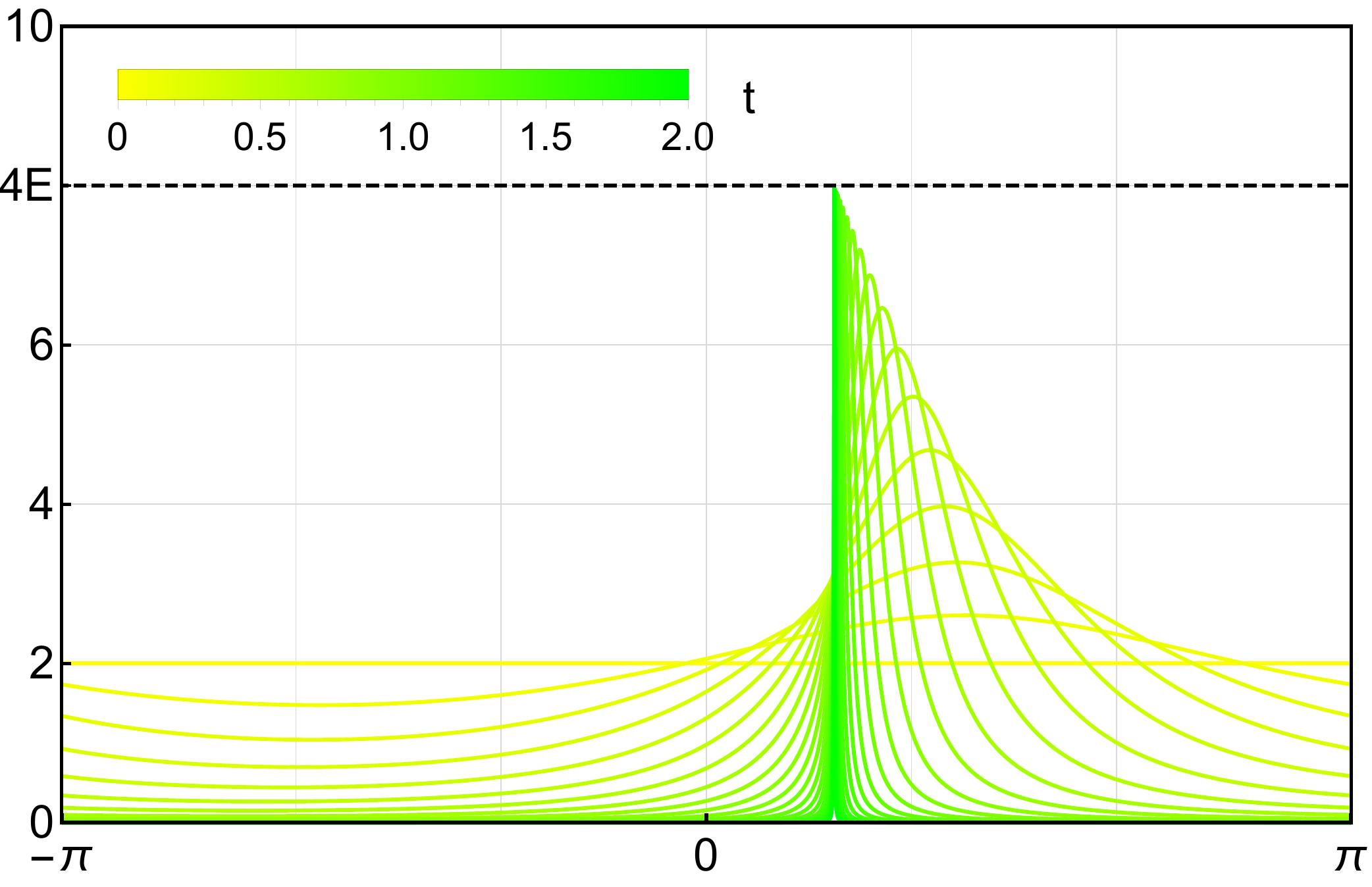}
		\caption{\small $\left(\theta,|v(t,\theta)|^2\right)$}
		\label{fig:collapse_wave_function_1}
	\end{subfigure}%
	\begin{subfigure}[b]{0.5\textwidth}
		\centering		\hspace{2cm}\includegraphics[width=8cm]{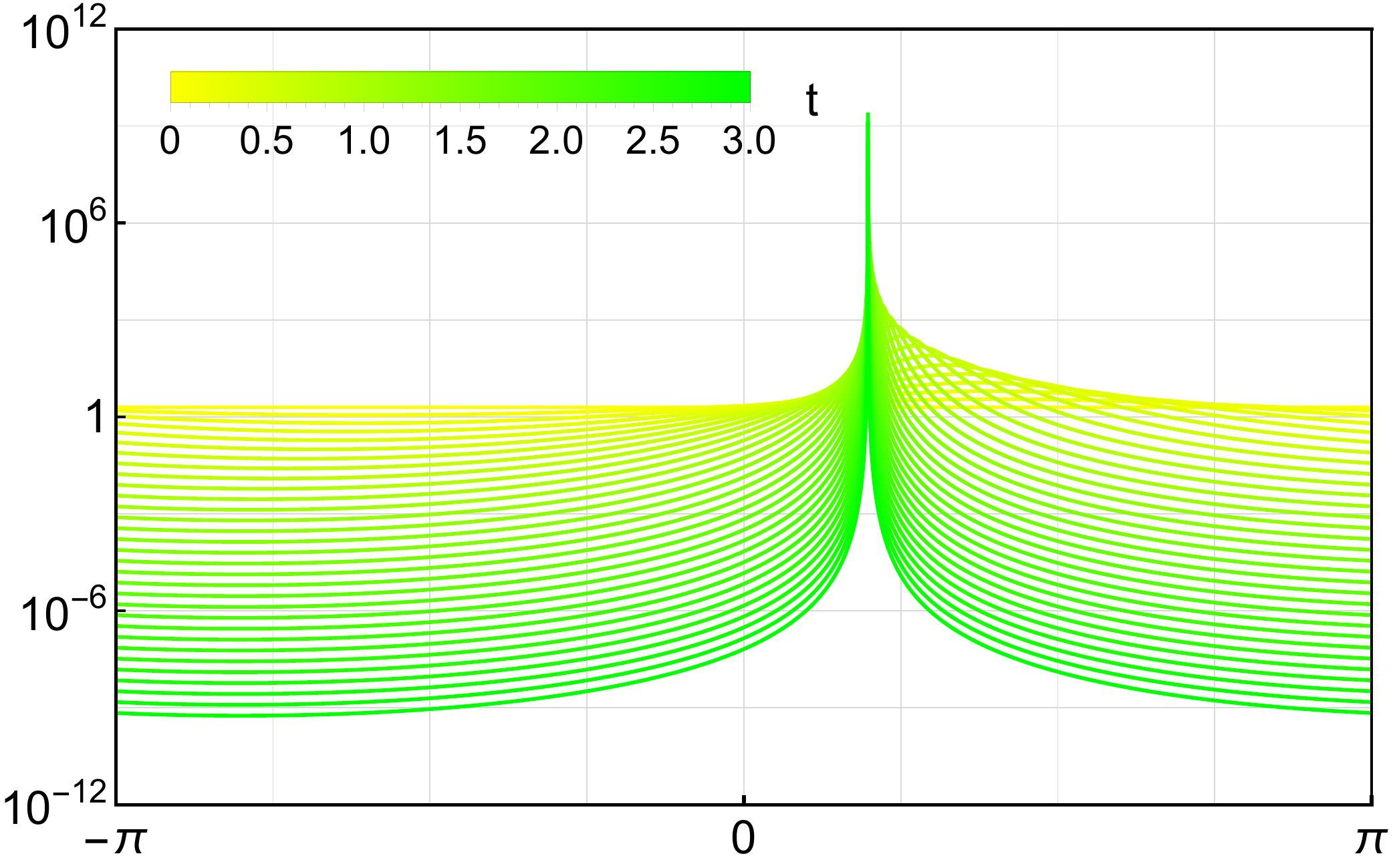}
		\caption{\small $\left(\theta,|\partial_\theta v(t,\theta)|^2\right)$}
		\label{fig:collapse_wave_function_2}
	\end{subfigure}%
		\caption{\small Concentration of the function $v(t,\theta)$ given by (\ref{eq:v_function}) and its derivative at a point for initial data within the family (\ref{eq:a0_general_ID}) characterized by exponential growth of Sobolev norms.  Fig.~\ref{fig:collapse_wave_function_1} shows that $|v(t,\theta)|^2$, despite becoming concentrated at a point, remains finite, converging to $4E$ there (black dashed line). Fig.~\ref{fig:collapse_wave_function_2} shows that the first derivative of $v(t,\theta)$ also becomes concentrated at a point, but in contrast with $v(t,\theta)$, its value at this point is not bounded.}
		\label{fig:collapse_wave_function}
	\end{figure}
	%%%%%%%%%%%%%%%%%%%%%%%%%%%%%%%%%%%%%%%%%%%%%%%%%%%%%%%%%%%%%%%%%%%%%%%%%%%%%%%%%%%%%%%%%%%%
	%%%%%%%%%%%%%%%%%%%%%%%%%%%%%%%%%%%%%%%%%%%%%%%%%%%%%%%%%%%%%%%%%%%%%%%%%%%%%%%%%%%%%%%%%%%%
	%%%%%%%%%%%%%%%%%%%%%%%%%%%%%%%%%%%%%%%%%%%%%%%%%%%%%%%%%%%%%%%%%%%%%%%%%%%%%%%%%%%%%%%%%%%%
	
	\section{The $\beta$-\Sz{} equation}\label{sec:beta_Szego}
	
	Since the truncated \Sz{} equation inherits many properties of the \cuSz, including a common Lax operator $K_u$ and a hierarchy of finite-dimensional invariant manifolds, it is natural to ask whether an interpolating family can be constructed connecting these two equations, retaining such special properties. To explore this question, we define
	\beq
	C_{nmkl}^{\left(\beta\right)} = \begin{cases}
		1 & \text{if } n m k l = 0\\
		1-\beta & \text{if } n m k l \neq 0
	\end{cases}
	\label{eq:C_beta}
	\eeq
	with $\beta \in \mathbb{R}$. Note that these coefficients are simply the linear combination
	\beq
	C^{(\beta)}=\beta C^{(tr)}+(1-\beta)C^{(\text{Sz})}.
	\label{eq:C_beta_linear_combination}
	\eeq
	Therefore, for $\beta = 0$ and $1$ we recover the original systems $C^{(0)} = C^{\left(\text{Sz}\right)}$ and $C^{(1)}~=~C^{\left(tr\right)}$. Additionally, in the limits $\beta \to \pm \infty$, one can rescale  $C^{\left(\beta\right)}$ by $\pm 1/\beta$ to obtain a ``shifted" \Sz{} system (the \cuSz{} with  $\{\alpha_0,\alpha_1,...\}$ replaced by $\{\alpha_1,\alpha_2,...\}$). Using the standard \Sz{} projector and  the shift operators, we can represent the  resonant system (\ref{eq:flow_eq_linear_spectrum}) with the interaction coefficients (\ref{eq:C_beta}) in position space as
	\beq
	i\partial_t u = \Pi(|u|^2u)-\beta S\Pi(|S^\dagger u|^2 S^\dagger u).
	\label{eq:beta_Szego_eq}
	\eeq
	We shall call this system the $\beta$-\Sz{} equation in analogy to the $\alpha$-\Sz{} equation, although there are important differences between these two deformations; see section \ref{sec:Discussion} for further discussion. The main properties of the $\beta$-\Sz{} equation, which we will analyze below, are:
	\begin{itemize}
		\item These systems possess one Lax pair
		\beq
		\frac{dK_{u}}{dt} = \left[C_u-\beta B_{S^\dagger u}, K_u\right].
		\label{eq:Lax_pair_beta_Szego}
		\eeq
		\item Sobolev norms of solutions of (\ref{eq:beta_Szego_eq}) subject to initial conditions $u_0 \in H_{+}^{s}\left(\mathbb{S}^1\right)$ with $s>1$, cannot grow faster than exponentially,
		\beq
			\|u(t)\|_{H^{s}} \leq \|u_0\|_{H^{s}} e^{C |t|}\qquad \text{with } C>0.
			\label{eq:beta_Szego_gobal_bound}
			\eeq
		\item There exist complex invariant manifolds $\mathcal{L}(D)$ given in (\ref{eq:L_odd}).
		\item For $\beta < 0$, the Sobolev norms for $u_0 \in \mathcal{L}(1)$ remain bounded,
		\beq
			\forall s \geq 0, \qquad \|u(t)\|_{H^{s}} \leq C, \qquad \text{with }\ C > 0.
			\label{eq:beta_Szego_bound_beta_less_zero}
		\eeq
		\item For $\beta > 0$, there exist $u_0 \in \mathcal{L}(1)$ such that the Sobolev norms with $s\hspace{-0.08cm} >\hspace{-0.08cm} \frac{1}{2}$ grow exponentially at late times,
		\beq
		\forall s > \frac{1}{2}, \qquad \|u(t)\|_{H^{s}} \underset{t\to\infty}{\simeq} e^{(2s-1)c|t|}.
		\label{eq:beta_Szego_exponential_bound}
		\eeq
		\item For $\beta \in (9,\infty)$, there exist $u_0 \in \mathcal{L}(1)$ such that Sobolev norms with $s\hspace{-0.08cm} >\hspace{-0.08cm} \frac{1}{2}$ have a polynomial growth at late times,
		\beq
		\forall s > \frac{1}{2}, \qquad \|u(t)\|_{H^{s}} \underset{t\to\infty}{\simeq} t^{(2s-1)}.
		\label{eq:beta_Szego_polynomial_bound}
		\eeq
		\item For $\beta \neq 1$, the $\beta$-\Sz{} equation contains the cubic \Sz{} equation as one of its invariant manifolds. It can be observed by restricting the initial conditions to odd modes (setting even modes to 0). Hence, if $\beta\neq 1$, the $\beta$-\Sz{} equation has subsectors with all the properties displayed in section \ref{subsec:Szego}.
	\end{itemize}
	A majority of these properties straightforwardly arise from the fact that the $\beta$-\Sz{} equation is a combination of the cubic \Sz{} and the truncated \Sz{} equations given by (\ref{eq:C_beta_linear_combination}). For the rest of our treatment, we shall focus on proving (\ref{eq:beta_Szego_bound_beta_less_zero}-\ref{eq:beta_Szego_polynomial_bound}), which requires explicit computations.\footnote{We shall not explore the dynamics in $u\in\mathcal{L}(D)$ for $D>1$. One may expect that these manifolds are integrable in the Liouville sense by analogy with \cite{GG,Xu,Thirouin}; namely, that they admit $2D+1$ conserved quantities in involution, the same as the number of dimensions. This is not a priori guaranteed by the Lax pair, and requires further analysis. The Lax pair provides $D$ conserved quantities, which in addition to the Hamiltonian and the $L^2$ norm give $D+2$ conservation laws. When $D>1$, this number by itself is insufficient for Liouville integrability. For the cubic \Sz{} equation, the remaining conserved quantities come from the second Lax pair (\ref{LaxSz}), while for the $\alpha$-\Sz{} and the quadratic \Sz{} equations they were found using different methods.}
	%%%%%%%%%%%%%%%%%%%%%%%%%%%%%%%%%%%%%%%%%%%%%%
	
	\subsection{Explicit blow-up in $\mathcal{L}(1)$}\label{subsec:Deformation_Szego_3d-manifold}

In order to  prove (\ref{eq:beta_Szego_bound_beta_less_zero}-\ref{eq:beta_Szego_polynomial_bound}) and provide explicit examples of such solutions,	
we essentially repeat our previous analysis of the truncated \Sz{} equation within $\mathcal{L}(1)$, but now at generic values of $\beta$. To this end, we write $\mathcal{L}(1)$ in the form (\ref{eq:ansatz_3d}), so that the equations of motion are reduced to
	\begin{align}
	i\dot{p} =& \left(N -\beta (1-|p|^2)E\right) p + a\bar{b},\\
	i\dot{b} =& (N+E)b + E a\bar{p},\\
	i\dot{a} = & \left(N - \beta E\right) a - \beta E |p|^2  b p.
	\end{align} 
	While $\beta$ appears on the right-hand side of these equations, the equation for $\dot{x}$ does not explicitly depend on this parameter (remember the definition $x(t) := |p(t)|^2$)
	\beq
	i\dot{x} = \left(a\bar b\bar p - \bar a b p\right).
	\eeq
	Expressions for $N$ and $E$ in (\ref{eq:E_J_3d_ansatz}) are also $\beta$-independent, and the only dependence comes from the Hamiltonian, which can be reduced to the conserved quantity
	\beq
	S = (N+E) x + \frac{\beta}{2} E x^2 + (a\bar b \bar p + \bar a b p),
	\eeq
	as was done in (\ref{eq:H_S}). In analogy with (\ref{eq:xdot_eq_2}), we make use of an effective potential $V_{\text{eff}}(x)$ to analyze the evolution of $x(t)$. It is a quartic polynomial for a generic $\beta$, but for specific values of this parameter it can be reduced to a cubic polynomial, or for $\beta = 0$ $\left(C^{\left(\text{Sz}\right)}\right)$, to a quadratic polynomial. Therefore, the family of models $C^{(\beta)}$ displays different phenomena for different values of $\beta$ and $p(0)$. Focusing our attention on initial conditions potentially displaying unbounded Sobolev norms, we observe that $V_{\text{eff}}(1)=0\Rightarrow V_{\text{eff}}'(1)=0$ and also that $V_{\text{eff}}(1) \geq 0$, preventing any blow-up in finite time, in agreement with the bound (\ref{eq:beta_Szego_gobal_bound}). As in the case of the truncated \Sz{} equation, turbulent solutions must satisfy $V_{\text{eff}}(1)=0$, a condition that, in terms of the conserved quantities, takes the form
	\beq
		(2+\beta)E - 2(N+S) = 0.
		\label{eq:S_func_N_E_beta}
	\eeq
	In terms of the dynamical variables,
	\beq
		2|b+a\bar p|^2 = \beta |bp+a|^2(1+|p|^2).
		\label{eq:constraint_V1_0_beta}
	\eeq
	It shows that, for $\beta < 0$, this conditions is not satisfied for nontrivial configurations and therefore, Sobolev norms of $u_0\in\mathcal{L}(1)$ remain bounded. For $\beta = 0$, the cubic \Sz{} system, this constraint is only satisfied for $b = -a\bar{p}$, which corresponds to stationary solutions as we will see later. The case $\beta > 0$ is different:  (\ref{eq:constraint_V1_0_beta}) is solved by 
	\beq
	a = b p \frac{\left(\left(1+|p|^2\right)\beta-2\right) +\sqrt{2\beta}e^{i\lambda} (1-|p|^2)\sqrt{1+\frac{1}{|p|^2}}}{(2-\beta)|p|^2 - \beta} \qquad \text{with } \lambda \in \mathbb{R},
	\label{eq:a_blow-up_beta}
	\eeq
	where certain combinations of $\beta$, $p$ and $\lambda$ lead to $\|u(t)\|_{H^{s}}\to\infty$ as $t\to \infty$ for $s>1/2$. Fig.~\ref{fig:phase_diagram_beta_x0} shows a sketch of the different regions in the $(\beta,p(0))$-plane with $\lambda = 0$ and $\pi$.
	\begin{figure}[t!]
		\centering
		\begin{subfigure}[b]{0.5\textwidth}
			\centering
			\includegraphics[width=8cm]{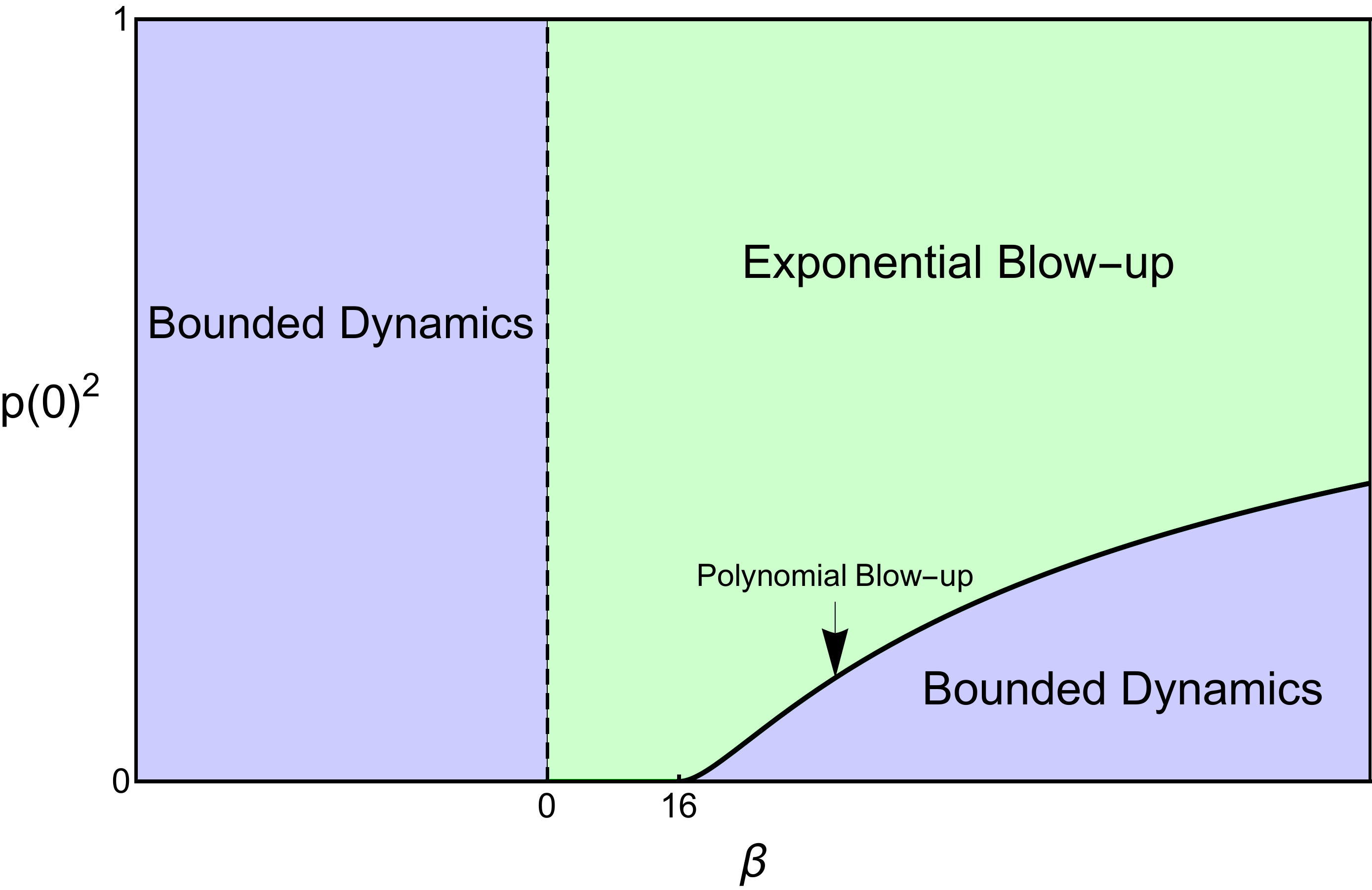}
			\caption{\small $\lambda = 0$}
			\label{fig:lambda_0}
		\end{subfigure}%
		\begin{subfigure}[b]{0.5\textwidth}
			\centering
			\includegraphics[width=8cm]{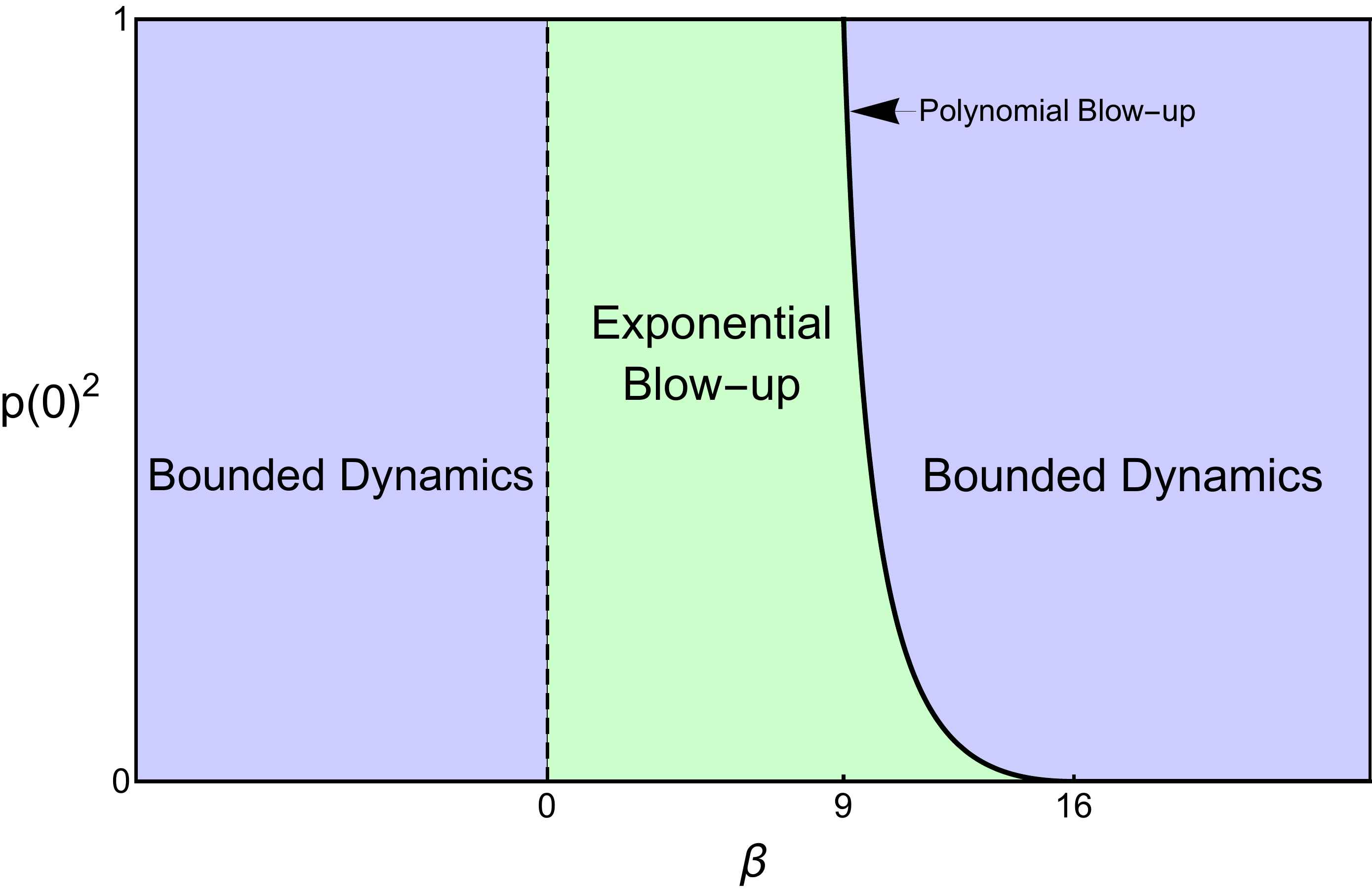}
			\caption{\small $\lambda = \pi$}
			\label{fig:lambda_pi}
		\end{subfigure}%
		\caption{\small Regions in the $(\beta,\ p(0))$-plane with different dynamics of initial conditions subject to (\ref{eq:a_blow-up_beta}) with $|b|\neq 0$ and $\lambda = 0,\pi$. Blue areas consist of initial configurations for which the Sobolev norms are bounded. Green areas consist of initial configurations with exponential growth of $\|u(t)\|_{H^{s}}$ for $s>1/2$. Black solid lines represent the boundary region with a polynomial growth. Black dashed lines are placed at $\beta= 0$, namely $C^{\left(\text{Sz}\right)}$, and demarcate a transition between systems with bounded and unbounded Sobolev norms. The pictures are naturally extended to infinity to the left and to the right.}
		\label{fig:phase_diagram_beta_x0}
	\end{figure}

	We shall now focus the discussion on the case $\lambda=0$, where the equations are simple enough to extract explicit expressions and this is a good representative of the behavior of the initial conditions (\ref{eq:a_blow-up_beta}) for generic $\lambda$.  In this case, under condition (\ref{eq:a_blow-up_beta}), the potential $V_{\text{eff}}(x)$ becomes
	\beq
	V_{\text{eff}}(x) = - |b(0)|^4 F \left(1-x\right)^2\left(x_{0} - x\right)\left(c - x\right),
	\eeq
	where $x_0 = x(0)$ and $F$ and $c$ are functions of $p(0)$ and $\beta$. For $0<\beta < 16$ we find that $V_{\text{eff}}(x)< 0$ for $x_0 < x < 1$ and Sobolev norms $\|u(t)\|_{H^{s}}$ with $s>1/2$ have exponential growth. For $\beta \geq 16$ there are three possibilities depending on $\beta$ and $p(0)$:
	\begin{itemize}
		\item  The additional zero $c \notin \left[x_0,1\right]$. In this case, the Sobolev norms with $s>1/2$\\ $\|u(t)\|_{H^{s}} \simeq e^{\left(2s - 1\right)\omega|t|/2}$ at late times.
		\item The additional zero $c \in [x_0,1)$. In this case, all Sobolev norms are bounded.
		\item The additional zero $c = 1$ (at the threshold between the two previous behaviors). In this case, the Sobolev norms with $s>1/2$ grow as $\|u(t)\|_{H^{s}} \simeq t^{\left(2s-1\right)}$ at late times.
	\end{itemize} 
	 It can be shown that the condition $c=1$ is solved by
	\beq
	x_0 = \frac{\left(\sqrt{\beta} - 4\right)^2}{\beta - 8}, \qquad \text{for }\ \beta \geq 16.
	\label{eq:x0_beta_polynomial}
	\eeq
	The explicit expressions for $x(t)$ as a function of time are then
	\beq
	x(t) = \begin{dcases}
		\frac{(c-x_0) \cosh \left(\omega t\right) - c - x_0 + 2 c x_0}{(c - x_0) \cosh\left(\omega t\right) + c + x_0 - 2} & \text{for } x_0 \in [0,1) \quad \text{if } 0 < \beta < 16 \\[15pt]
		\frac{(c-x_0) \cos \left(\omega t\right) - c - x_0 + 2 c x_0}{(c - x_0) \cos\left(\omega t\right) + c + x_0 - 2} & \text{for } x_0 \in \big{[}0,\frac{\left(\sqrt{\beta} - 4\right)^2}{\beta - 8}\bigg{)} \quad \text{if } \beta > 16 \\[15pt]
		\hspace{2.5cm}\frac{ \tilde{c} t^2 + x_0}{\tilde{c} t^2 + 1} & \text{for } x_0 = \frac{\left(\sqrt{\beta} - 4\right)^2}{\beta - 8} \quad \text{if } \beta \geq 16 \\[15pt]
		\frac{(c-x_0) \cosh \left(\omega t\right) - c - x_0 + 2 c x_0}{(c - x_0) \cosh\left(\omega t\right) + c + x_0 - 2} & \text{for } x_0 \in \left(\frac{\left(\sqrt{\beta} - 4\right)^2}{\beta - 8},1\right) \quad \text{if } \beta \geq 16
	\end{dcases}
	\eeq 
	where $\omega = |b(0)|^2|\sqrt{F(1-c)(1-x_0)}|$ and $\tilde{c} = - |b(0)|^4 F (1-x_0)^2/4$ (with $\tilde{c} > 0$ for (\ref{eq:x0_beta_polynomial})).
	
	 As can be observed in fig.~\ref{fig:phase_diagram_beta_x0}, for $\lambda = \pi$, we reach to similar conclusions, but in this case the transition ($c=1$) between bounded and unbounded Sobolev norms is placed at 
	\beq
	x_0 = \frac{\left(\sqrt{\beta} - 4\right)^2}{\beta - 8}, \qquad \text{for }\ \beta \in (9,16],
	\label{eq:x0_beta_polynomial_v2}
	\eeq
	and also displays polynomial growth of Sobolev norms. For generic $\lambda$, the scenario is the same: we find combinations of the parameters $\beta$ and $p(0)$ that lead to exponential growth of Sobolev norms, other combinations for which these norms are bounded, and a curve (depending on $\lambda$), separating the two previous behaviors, that displays polynomial growth of Sobolev norms. Note that the cases $\lambda= 0$ and $\pi$ show that $\beta\in(9,\infty)$ is sufficient for the existence of initial conditions $u_0\in\mathcal{L}(1)$ with a polynomial growth of Sobolev norms. This range of values of $\beta$ is also necessary. It can be seen after imposing the appropriate conditions over $V_{\text{eff}}(x)$; namely, $V_{\text{eff}}(1) = V_{\text{eff}}''(1) = V_{\text{eff}}(x_0) = 0$ and also $V_{\text{eff}}(x) < 0$ for $x\in(x_0,1)$ with $x_0 \in[0,1)$. As we explained above, the first condition is satisfied by (\ref{eq:S_func_N_E_beta}). For the second condition, after expressing $V_{\text{eff}}(x)$ in terms of $N$, $E$ and $x$ with (\ref{eq:S_func_N_E_beta}), we obtain the relation
		\beq
			N = (1\pm \sqrt{\beta})^2 E.
			\label{eq:N_func_E_beta}
		\eeq
		Due to (\ref{eq:S_func_N_E_beta}) and (\ref{eq:N_func_E_beta}), the above conditions involving $x_0$ are only satisfied for $\beta > 9$. 
		  
	\begin{figure}[t!]
		\centering
		\begin{subfigure}[b]{0.5\textwidth}
			\centering
			\includegraphics[width=8cm]{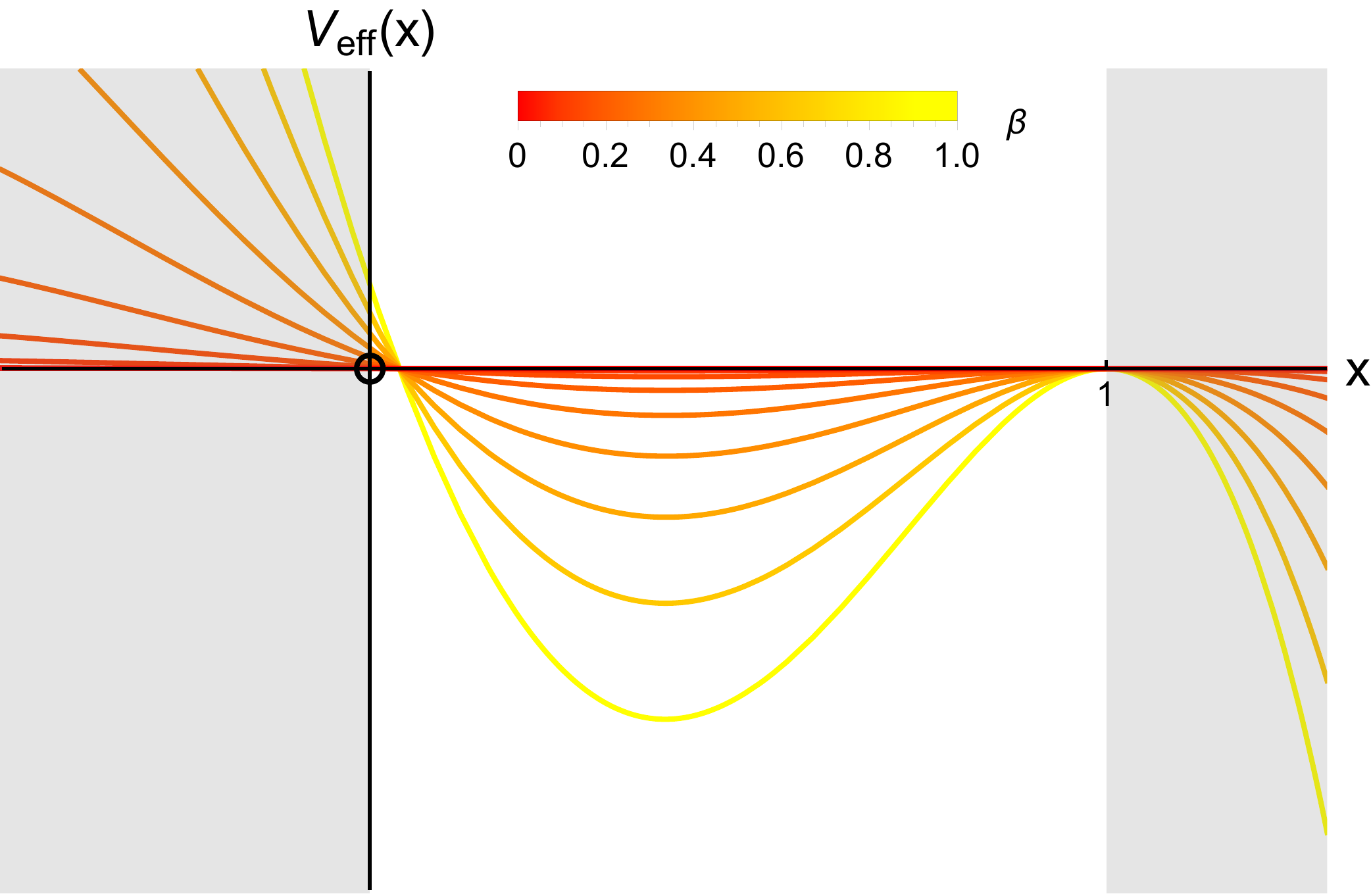}
			\caption{\small $\beta \geq 0$}
			\label{fig:Veff_Limit_beta_0_b}
		\end{subfigure}%
		\begin{subfigure}[b]{0.5\textwidth}
			\centering
			\includegraphics[width=8cm]{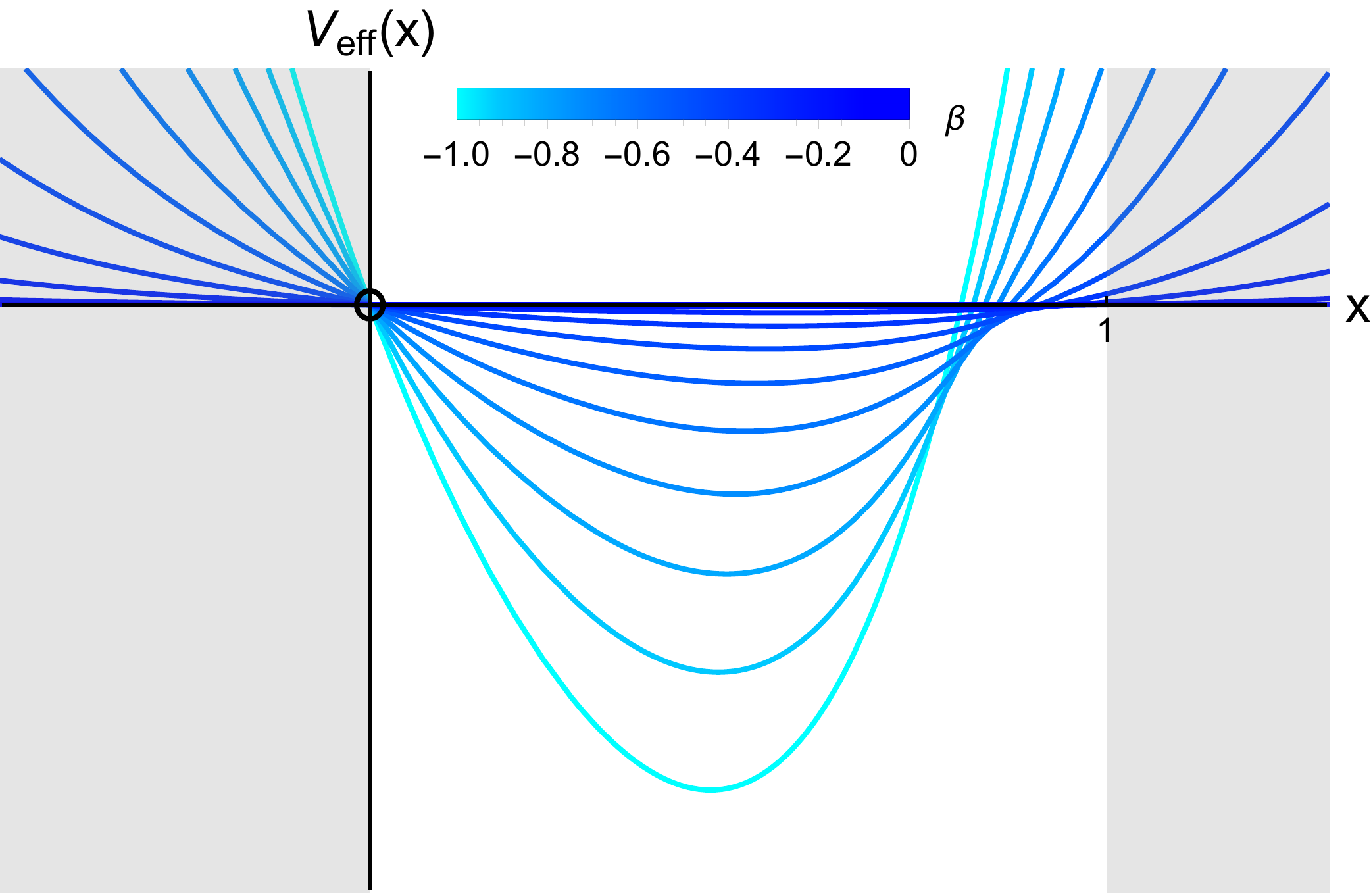}
			\caption{\small $\beta \leq 0$}
			\label{fig:Veff_Limit_beta_0_a}
		\end{subfigure}%
		\caption{\small Effective potential for the initial configuration (\ref{eq:a_blow-up_beta}) with $b = (2-\beta)p\bar{p} - \beta$, $p=1/5$. For positive $\beta$, the potential has zeros at $0$ and $1$; unlike negative $\beta$, where it has zeros at $0$ and at a value less than 1. These plots show that $V_{\text{eff}}(x)$ converges to the x-axis for $\beta \to 0$. This behavior has been observed for generic $|p|\in [0,1)$, indicating the existence of a family of stationary solutions at the threshold $\left(C^{\left(\text{Sz}\right)}\right)$ of models with bounded and unbounded Sobolev norms. The center of the black ring is placed at $(0,0)$ and the shaded areas do not correspond to valid configurations of the dynamical variables.}
		\label{fig:Veff_Limit_beta_0}
	\end{figure}
	
	As we can see in fig.~\ref{fig:phase_diagram_beta_x0}, the \cuSz{} forms a transition point between systems with bounded and unbounded Sobolev norms. It is natural to wonder which solutions of \cuSz{} separate these behaviors. Fig.~\ref{fig:Veff_Limit_beta_0} provides a visual illustration of the transition. Setting $\beta = 0$ and $b = -\bar p$, the initial condition (\ref{eq:a_blow-up_beta}) reduces to a family of stationary solutions for the \cuSz{} \cite{GG}. Its expression is more recognizable from the literature in terms of the generating function
	\beq
	u(t,z) = e^{-i t}\frac{\bar p - z}{1 - p z}, \qquad \text{with }\ u(t,z) = \sum_{n=0}^{\infty} \alpha_n(t) z^n.
	\label{eq:stationary_solution_Szego}
	\eeq
	This result makes us wonder whether one can find other modifications of the \cuSz{} such that (\ref{eq:stationary_solution_Szego}) changes into a turbulent solution.

	In addition to our analysis of the $\beta$-\Sz{} equation, we present, in appendix B, an even bigger family of deformations of the \cuSz{} that retain some of its Lax integrability and invariant manifolds and exhibit unbounded Sobolev norm growth. The mode couplings are of the following form (and they can be freely combined with an arbitrary $\alpha$-deformation):
\beq
\begin{cases}
	C_{0000} = \gamma, &\\
	C_{n0n0} = \delta_1 + \delta_2 n, & \text{for } n\neq 0,\\
	C_{nmkl} = 1, & \text{other cases with } nmkl = 0,\\
	C_{nmkl} = 1-\beta, & nmkl \neq 0.
\end{cases}	
\label{eq:delta_gamma_deformations_v2}
\eeq

It would be interesting to study more systematically which deformations of the \cuSz{} respect the invariant manifolds $\mathcal{L}\left(D\right)$ or admit a Lax pair based on the operator $K_u$, but we shall not pursue it here. Similar type of analysis was performed in \cite{BBE} for a related question, namely, which resonant systems of the form (\ref{eq:flow_eq_linear_spectrum}) respect another explicitly defined three-dimensional invariant manifold, different from $\mathcal{L}\left(1\right)$.

	%%%%%%%%%%%%%%%%%%%%%%%%%%%%%%%%%%%%%%%%%%%%%%%%%%%%%%%%%%%%%%%%%%%%%%%%%%%%%%%%%%%%%%%%%%%%
	%%%%%%%%%%%%%%%%%%%%%%%%%%%%%%%%%%%%%%%%%%%%%%%%%%%%%%%%%%%%%%%%%%%%%%%%%%%%%%%%%%%%%%%%%%%%
	%%%%%%%%%%%%%%%%%%%%%%%%%%%%%%%%%%%%%%%%%%%%%%%%%%%%%%%%%%%%%%%%%%%%%%%%%%%%%%%%%%%%%%%%%%%%
	
	\section{Discussion and outlook}\label{sec:Discussion}

We have presented a large family of modifications of the \cuSz{} beyond the $\al$-\Sz{} equation of \cite{Xu} that retain its Lax pair structure
and a hierarchy of finite-dimensional dynamically invariant manifolds. A central role in this family is played by the truncated \Sz{} system (\ref{eq:resonant_eq_BE}-\ref{eq:truncated_Szego_eq}), where a majority of the Fourier mode couplings present in the original \cuSz{} have been eliminated. The systems we have introduced can be explicitly analyzed within the simplest three-dimensional invariant manifold given by (\ref{eq:ansatz_3d}), and display a variety of turbulent cascades, including unbounded exponential or polynomial growth of Sobolev norms.
These cascades are stronger than what is seen in the original \cuSz, which is particularly striking for the truncated \Sz{} equation, since naively, one would imagine that eliminating couplings
between different sets of Fourier modes should weaken rather than strengthen turbulent cascades. One is thus encouraged to rethink the role played
by mode couplings in turbulent phenomena.
	
	Our systems display parallels to other deformations of the \cuSz{} exhibiting unbounded Sobolev norm growth, such as the $\alpha$-\Sz{} \cite{Xu} and the damped \Sz{} \cite{GG3} equations; nevertheless, there are significant differences. We highlight the main differences between the $\alpha$-\Sz{} and the $\beta$-\Sz{} equations (similar remarks could be made about the damped \Sz{} equation, which is further apart from our models):
	\begin{itemize}
		\item The $\alpha$-\Sz{} model (\ref{eq:alpha_Szego_modes}) introduces a deformation in the linear part of the equation for the lowest mode, while the $\beta$-\Sz{} model implements a modification of the cubic part, keeping the system within the resonant class  (\ref{eq:flow_eq_linear_spectrum}).
		\item The $\al$-deformation, or its extension in (\ref{eq:alpha_deformation_extension}), only explicitly affect the lowest mode. The $\beta$-deformation nontrivially modifies the equations of motion for all modes.
		\item After an appropriate rescaling, the $\alpha$-\Sz{} equation is reduced to three relevant systems (\ref{eq:alpha_Szego_v2}) only depending on $\text{sgn}(\alpha)$. In the case of the $\beta$-\Sz{} equation, changes in $\beta$ cannot be absorbed into rescaling, leaving an essentially continuous family of systems.
		\item No explicit solutions with polynomial growth of Sobolev norms are known for the $\alpha$-\Sz{}, specifically for initial conditions in $\mathcal{L}\left(1\right)$. Such solutions are seen for the $\beta$-\Sz{} system for some values of the parameters, see fig.~\ref{fig:phase_diagram_beta_x0}.
		\item The $\alpha$ and $\beta$ deformations can be implemented simultaneously, together with a few further deformations described in appendix B.
	\end{itemize}

We would like to make a further  brief digression that highlights, from a perspective rather different from our main treatment, the distinction between the $\alpha$-\Sz{} equation and the $\beta$-deformations, in particular,
the truncated \Sz{} equation. All of the systems we have considered here are Hamiltonian, and the standard procedure of {\it quantization} may be applied to such systems, according to the basic principles of quantum mechanics. The generalities of quantization of resonant systems of the form 
(\ref{eq:flow_eq_linear_spectrum}) have been considered in \cite{EP}, with connections to the extensive lore of the quantum chaos theory \cite{QCh}. One then studies the corresponding quantum energy spectra, which are in turn
expected to display different distributions of distances between neighboring levels, depending on the integrability properties of the system. By doing so, one discovers that the cubic \Sz{} system is extremely special, displaying a purely integer energy spectrum \cite{EP}, while a generic integrable system is expected
to display a Poissonian distribution of energy level distances \cite{QCh}. One is thus led to believe that the \cuSz{} possesses structure beyond ordinary integrability
(an explicit example of that is two inequivalent Lax pairs, as opposed to just one). If one turns on the $\al$-deformation (or the related $\de$-deformations from appendix B), the quantum energy spectrum is no longer integer, but the distribution of energy level distances is nowhere close to Poissonian, with too many small energy level gaps. On the other hand, the truncated \Sz{} system, in its quantum version, displays a perfectly
Poissonian distribution of energy level spacings, making it an excellent candidate for a generic Lax-integrable system within the resonant class  (\ref{eq:flow_eq_linear_spectrum}). This may make the truncated \Sz{} system an attractive playground for quantum chaos and integrability studies of the type undertaken in \cite{EP}, quite far from the topics that initially stimulated our search for this system.
	
	In all deformations considered in our treatment, a special role is played by mode 0. One could ask what happens if this role is swapped with another mode. 
For example, instead of (\ref{eq:C_beta}), we could consider the following modification of the interactions that do not involve mode 1:
	\beq
	C_{nmkl} = \begin{cases}
		1 & \text{if } (n-1) (m-1) (k-1) (l-1) = 0,\\
		1-\beta & \text{if } (n-1) (m-1) (k-1) (l-1) \neq 0.
	\end{cases}
	\label{eq:C_mode_1}
	\eeq
While we did not analyze the general properties of this system, we know that, after restricting the initial conditions to odd modes, the dynamics of (\ref{eq:C_mode_1}) is governed by $C^{\left(\beta\right)}$. Hence, for $\beta> 0$ it has solutions with unbounded Sobolev norms. This trivial argument can be extended to other similar deformations of the \Sz{} equation anchored on other modes. 

We conclude with a `teaser' regarding finite-time turbulent blow-up, a question that has indirectly led us to the main discoveries presented in this article.
Finite-time turbulent blow-up is known (from numerical simulations) to take place in extremely complicated (and physically interesting) resonant systems
emerging in Anti-de Sitter spacetimes \cite{FPU,CEV,BMR}. It would be very desirable to have a simple explicit resonant system in which this
phenomenon can be analyzed. Finite-time blow-up cannot happen in the \cuSz, or any of the other systems considered in this article, on account
of the exponential upper bounds on Sobolev norm growth. Our numerical experiments indicate, however, that finite-time turbulent blow-up does
happen in a few simple closely related systems within the resonant class (\ref{eq:flow_eq_linear_spectrum}). More specifically, we have considered
interaction coefficients of the form
	\beq
	C_{nmkl} = (n+1)^G(m+1)^G(k+1)^G(l+1)^G \qquad \text{with }G>0,
	\eeq
	(note that $G=0$ is the \cuSz), as well as a truncated version of these systems analogous to the truncated \Sz{} equation (if all of the indices are non-zero, the corresponding $C$ is replaced by zero, otherwise it remains intact). We  have focused on numerical simulations of two-mode initial data
	\beq
	|\alpha_0(0)| \neq 0, \quad |\alpha_1(0)| \neq 0, \quad |\alpha_{n\geq 2}(0)| = 0
	\eeq
for the cases $G = 1/2$ and $1$.
	The blow-up manifests itself as the following asymptotic behavior at large $n$:
	\beq
	\alpha_{n \gg 1}(t) \sim c(t) n^{-\gamma} e^{-\rho(t)n},
	\label{eq:asymptotic_ansatz}
	\eeq
	with $\rho(t) \to 0$ as $t\to t^* < \infty$ and $\gamma=2$ ($\gamma = 5/2$) for $G=1/2$ (for $G=1$). 
	Other systems that we have considered and observed similar phenomena are
	\beq
	C_{nmkl} = (n+m+1)^{G} \qquad \text{with } G> 0,
	\eeq
	as well as their truncated versions. These strong and simple blow-up behaviors beg for an analytic explanation.

		%%%%%%%%%%%%%%%%%%%%%%%%%%%%%%%%%%%%%%%%%%%%%%%%%%%%%%%%%%%%%%%%%%%%%%%%%%%%%%%%%%%%%%%%%%%%
	%%%%%%%%%%%%%%%%%%%%%%%%%%%%%%%%%%%%%%%%%%%%%%%%%%%%%%%%%%%%%%%%%%%%%%%%%%%%%%%%%%%%%%%%%%%%
	%%%%%%%%%%%%%%%%%%%%%%%%%%%%%%%%%%%%%%%%%%%%%%%%%%%%%%%%%%%%%%%%%%%%%%%%%%%%%%%%%%%%%%%%%%%%
	%%%%%%%%%%%%%%%%%%%%%%%%%%%%%%%%%%%%%%%%%%%%%%%%%%%%
	\section*{Acknowledgments}

We are indebted to Piotr Bizo\'n for collaboration on related subjects, numerious discussions, and comments on the manuscript.
This work has been supported by the Polish National Science Centre grant number 2017/26/A/ST2/00530, by CUniverse research promotion project at Chulalongkorn University (grant CUAASC), and by FWO-Vlaanderen through project G006918N.
	
	%%%%%%%%%%%%%%%%%%%%%%%%%%%%%%%%%%%%%%%%%%%%%%%%%%%%%%%%%%%%%%%%%%%%%%%%%%%%%%%%%%%%%%%%%%%%
	%%%%%%%%%%%%%%%%%%%%%%%%%%%%%%%%%%%%%%%%%%%%%%%%%%%%%%%%%%%%%%%%%%%%%%%%%%%%%%%%%%%%%%%%%%%%
	%%%%%%%%%%%%%%%%%%%%%%%%%%%%%%%%%%%%%%%%%%%%%%%%%%%%%%%%%%%%%%%%%%%%%%%%%%%%%%%%%%%%%%%%%%%%
	%%%%%%%%%%%%%%%%%%%%%%%%%%%%%%%%%%%%%%%%%%%%%%%%%%%%
%	\begin{appendices}
	\appendix
		
	\section*{Appendix A: Elementary analysis of the Lax pair structure}
	\label{appendix:_Lax_pair}

	The Lax pair structure can be re-expressed through only elementary operations acting on the sequences $\{\al_n\}$ and $\{h_n\}$. We would like to verify the Lax pair (\ref{DK}) of the truncated \Sz{} equation
(\ref{eq:resonant_eq_BE}) in this language. To this end, we recall the expression for $u$ given by (\ref{eq:u_variable}) 
	and the operators (\ref{eq:H_T_S_operators}), whose action in components is
	$$
	(H_u h)_n=\sum_{m=0}^\infty \al_{n+m}\bar h_m,\qquad (T_b h)_n=\sum_{m=0}^\infty b_{n-m} h_m, \qquad (Sh)_n=h_{n-1}, \qquad (S^\dagger h)_n=h_{n+1}.
	$$
	We also have the operators
	$$
	K_u = S^{\dagger}H_u, \qquad D_u=-iT_{|u|^2-|S^\dagger u|^2}.
	$$
	 We first note that
	$$
	|u|^2-|S^\dagger u|^2=\sum_{k,l=0}^{\infty}\alb_k\al_l e^{i(l-k)\te}-\sum_{k,l=1}^{\infty}\alb_k\al_l e^{i(l-k)\te}=\alb_0\sum_{k=0}^\infty \al_k e^{ik\te }+\al_0\sum_{k=1}^\infty \alb_k e^{-ik\te }.
	$$
	Using this expression, the action of the Lax pair on a test vector $h_n$ with $n\ge 0$ can be written as
	$$
	[K_uh]_n=\sum_{m=0}^\infty \al_{n+m+1} \bar h_m,\qquad [D_uh]_n=-i\alb_0\sum_{k=0}^n\al_{k}h_{n-k}-i\al_0\sum_{k=1}^\infty \bar{\alpha}_k h_{n+k}.
	$$
	Then, since $K$ does not depend on $\al_0$, one can use the equation of motion for $\al_{n\ge 1}$
	$$
	i\dot\al_n=\alb_0\sum_{k=0}^n \al_k \al_{n-k} +2\al_0\sum_{k=1}^\infty \alb_{k}\al_{n+k},
	$$
	and the antilinearity of $K_u$ to write
	\begin{align*}
	i[(\dot K_u - D_uK_u + K_uD_u)h]_n&=\alb_0 \sum_{m=0}^\infty \sum_{k=0}^{n+m+1}\al_{n+m+1-k}\al_k\bar h_m
	+2\al_0\sum_{m=0}^\infty\sum_{k=1}^\infty \alb_k\al_{n+m+k+1}\bar h_m\nonumber\\
	&-\alb_0\sum_{k=0}^n\sum_{m=0}^\infty \al_k\al_{n-k+m+1}\bar h_m
	-\al_0\sum_{k=1}^\infty\sum_{m=0}^\infty \alb_k\al_{n+m+k+1}\bar h_m\nonumber\\
	&-\al_0\sum_{m=0}^\infty\sum_{k=0}^m\al_{n+m+1}\alb_{k}\bar h_{m-k}
	-\alb_0\sum_{m=0}^\infty\sum_{k=1}^\infty\al_{n+m+1}\al_k\bar h_{m+k}.\nonumber
	\end{align*}
	We transform the first term in the last line as
	$$
	\sum_{m=0}^\infty\sum_{k=0}^m\al_{n+m+1}\alb_{m-k}\bar h_{k}=\sum_{k=0}^\infty\sum_{m=k}^\infty \al_{n+m+1}\alb_{m-k}\bar h_{k}=\sum_{k,m=0}^\infty \al_{n+m+k+1}\alb_{m}\bar h_{k},
	$$
	and the last term as
	\begin{align*}
	&\sum_{k=1}^\infty\sum_{m=0}^\infty\al_{n+m+1}\al_k\bar h_{m+k}=\sum_{k=1}^\infty \sum_{m=k}^\infty\al_{n+m-k+1}\al_k\bar h_{m}\\
&\hspace{5cm}=\sum_{m=1}^\infty \sum_{k=1}^m \al_{n+m-k+1}\al_k\bar h_{m}=\sum_{m=1}^\infty \sum_{k=n+1}^{n+m}\al_k \al_{n+m-k+1}\bar h_{m}.
	\end{align*}
	Then, combining all the terms proportional to $\alb_0$ in the above expression for $i[(\dot K_u - D_uK_u + K_uD_u)h]_n$ yields
	\begin{align*}
	\alb_0&\left(\sum_{m=0}^\infty \sum_{k=0}^{n+m+1}\al_{n+m+1-k}\al_k\bar h_m
	-\sum_{k=0}^n\sum_{m=0}^\infty \al_k\al_{n-k+m+1}\bar h_m
	-\sum_{m=1}^\infty \sum_{k=n+1}^{n+m}\al_k \al_{n+m-k+1}\bar h_{m}\right)\nonumber\\
	&=|\al_0|^2\sum_{m=0}^\infty \al_{n+m+1}\bar h_m,\nonumber
	\end{align*}
	while combining all the terms proportional to $\al_0$ yields
	\begin{align*}
	\al_0&\left(2\sum_{m=0}^\infty\sum_{k=1}^\infty \alb_k\al_{n+m+k+1}\bar h_m
	-\sum_{k=1}^\infty\sum_{m=0}^\infty \alb_k\al_{n+m+k+1}\bar h_m
	-\sum_{k,m=0}^\infty \al_{n+m+k+1}\alb_{k}\bar h_{m}\right)\nonumber\\
	&=-|\al_0|^2\sum_{m=0}^\infty \al_{n+m+1}\bar h_m.\nonumber
	\end{align*}
	Altogether,
	$$
	\frac{dK_u}{dt}=[D_u,K_u],
	$$
	so the validity of the Lax pair has been verified.

	%%%%%%%%%%%%%%%%%%%%%%%%%%%%%%%%%%%%%%%%%%%%%%%%%%%%%%%%%%%%%%%%%%%%%%%%%%%%%%%%%%%%%%%%%%%%
	%%%%%%%%%%%%%%%%%%%%%%%%%%%%%%%%%%%%%%%%%%%%%%%%%%%%%%%%%%%%%%%%%%%%%%%%%%%%%%%%%%%%%%%%%%%%
	%%%%%%%%%%%%%%%%%%%%%%%%%%%%%%%%%%%%%%%%%%%%%%%%%%%%%%%%%%%%%%%%%%%%%%%%%%%%%%%%%%%%%%%%%%%%
	
	\section*{Appendix B: Further deformations of the cubic \Sz{} equation}
	\label{sec:additional_deformations}
	We shall now briefly address additional deformations of the \cuSz, besides the $\alpha$-deformation (\ref{eq:alpha_Szego_v1}) proposed in \cite{Xu} and the $\beta$-deformation we introduced in section~\ref{sec:beta_Szego}, that preserve the relevant features such as Lax integrability and invariant manifolds, and/or exhibit unbounded Sobolev norm growth. We can propose an explicit  family of the interaction coefficients (\ref{eq:delta_gamma_deformations_v2}) that meets such demands, which can be freely combined with the $\al$-deformation, given by a single linear term as in (\ref{eq:alpha_Szego_v1}). One could engineer some other modifications of the cubic terms or the linear part; however, they can be reduced to the ones above and the $\alpha$-deformation by scaling and the transformation $\alpha_n \to e^{i (\theta_1 +\theta_2 n)t}$. Thus, the number of independent deformations for the \cuSz{} that we consider here is five, one for the linear part, $\alpha$, and four for the cubic terms (\ref{eq:delta_gamma_deformations_v2}). We are not going to analyze these models in detail; nevertheless, following the procedure of the previous sections, one can arrive at the properties listed below:
	\begin{itemize}
		\item The five deformations $\alpha,\beta,\gamma,\delta_1$ and $\delta_2$ have the following position space representation, with $u$ given by (\ref{eq:u_variable}):
		\begin{align}
		i\dot u = &\Pi(|u|^2u)-\beta S\Pi(|S^\dagger u|^2 S^\dagger u) + \alpha \left(u|1\right) + \tilde\gamma|(u|1)|^2 (u|1) 
		\label{eq:shenlong_Szego}\\
		&+ 2|(u|1)|^2(\tilde\delta_1-i\tilde\delta_2 \partial_\te)u+ 2(u|1)(\tilde\delta_1|u|^2-i\tilde\delta_2 \bar{u} \del_{\theta}u|1),\nonumber
		\end{align}
		where the tildes over the parameters indicate that linear redefinitions have been made compared to the parameters introduced in (\ref{eq:delta_gamma_deformations_v2}).
		
		\item The system (\ref{eq:shenlong_Szego}) admits the following Lax pair structure ($\mathbb{I}$ denotes the identity):
		\beq
		\frac{dK_{u}}{dt} = \left[A_u, K_u\right],
		\eeq
		\beq
		A_u = C_u-\beta B_{S^\dagger u} - i\tilde\delta_1|(u|1)|^2\mathbb{I} - i \tilde\delta_2 |(u|1)|^2 (-2i\del_{\theta} + \mathbb{I}).
		\eeq 
		Note that, when verifying the Lax pair, the last two terms in the first line of (\ref{eq:shenlong_Szego}) and the second half of the second line do not contribute to $dK_u/dt$, since they are annihilated by $S^\dagger$.
		\item There exist complex invariant manifolds $\mathcal{L}\left(D\right)$ given in (\ref{eq:L_odd}).
		
		\item There are values of parameters such that Sobolev norms of any $u_0\in\mathcal{L}\left(1\right)$ remain bounded,
		\beq
		\forall s\geq0, \|u(t)\|_{H^s}\leq C, \quad \text{with } C>0.
		\eeq
		\item There are values of parameters such that, for some $u_0\in\mathcal{L}\left(1\right)$, Sobolev norms with $s>1/2$ are unbounded,
		\beq
		\forall s>\frac{1}{2}, \qquad \|u(t)\|_{H^s} \underset{t\to\infty}{\to}\infty.
		\eeq
		
		\item The $\gamma$ deformation is a particular case of a more general deformation briefly mentioned by Xu at the end of the second reference in \cite{Xu}, where the Hamiltonian 
		\beq
		\Ham = \frac{1}{4}\|u\|_{L^4}^{4} + \frac{1}{2} F\left(|\left(u|1\right)|^2\right)
		\label{eq:alpha_deformation_extension}
		\eeq
		is proposed. Our case corresponds to $F(x) = \gamma x^2/4$, which is the same as modifying $C_{0000}$ as in (\ref{eq:delta_gamma_deformations_v2}), and the result respects the general structure of resonant systems (\ref{eq:flow_eq_linear_spectrum}). To the best of our knowledge, turbulent properties of this model have not been previously investigated. 
		
		\item The $\de_1$ and $\de_2$ deformations can be thought of as an $(N,E)$-dependent redefinition\\ $\alpha\to \alpha+\de_1 N+\de_2 E$ in the $\alpha$-deformation, where $N$ and $E$ are the conserved quantities given by (\ref{eq:N_consered_quantity}-\ref{eq:E_consered_quantity}). Due to the conservation of $N$ and $E$, this essentially amounts to relabelling the trajectories of the $\alpha$-\Sz{} system at different values of $\al$. Note, however, that the $\de_1$ and $\de_2$ deformations keep the system within the resonant class (\ref{eq:flow_eq_linear_spectrum}), while the $\al$-deformation does not.
	\end{itemize}
	
	It would be interesting to study more systematically which deformations of the \cuSz{} respect the invariant manifolds $\mathcal{L}\left(D\right)$ or admit a Lax pair based on the operator $K_u$, but we shall not pursue it here. Similar type of analysis was performed in \cite{BBE} for a related question, namely, which resonant systems of the form (\ref{eq:flow_eq_linear_spectrum}) respect another explicitly defined three-dimensional invariant manifold, different from $\mathcal{L}\left(1\right)$.

%	\end{appendices}
	
	%%%%%%%%%%%%%%%%%%%%%%%%%%%%%%%%%%%%%%%%%%%%%%%%%%%%%%%%%%%%%%%%%%%%%%%%%%%%%%%%%%%%%%%%%%%%
	%%%%%%%%%%%%%%%%%%%%%%%%%%%%%%%%%%%%%%%%%%%%%%%%%%%%%%%%%%%%%%%%%%%%%%%%%%%%%%%%%%%%%%%%%%%%
	%%%%%%%%%%%%%%%%%%%%%%%%%%%%%%%%%%%%%%%%%%%%%%%%%%%%%%%%%%%%%%%%%%%%%%%%%%%%%%%%%%%%%%%%%%%%


\begin{thebibliography}{99}

\bibitem{Carlton}J.~S.~Carlton, {\em The early development of the screw propeller}, in {\em Marine propellers and propulsion} (Elsevier, 2019).
		
		\bibitem{Bourgain} J.~Bourgain, {\em Problems in Hamiltonian PDEs}, in {\em Visions in mathematics} (Birkh\"auser Basel, 2010).

\bibitem{Nazarenko} S.~Nazarenko, {\em Wave turbulence} (Springer, 2011).

\bibitem{Kuksin} S.~B.~Kuksin, {\em Oscillations in space-periodic nonlinear Schr\"odinger equations}, Geom. Funct. Anal. {\bf 7} (1997) 338.

\bibitem{Colliander} J.~Colliander, M.~Keel, G.~Staffilani, H.~Takaoka and T.~Tao, {\em Transfer of energy to high frequencies in the cubic defocusing nonlinear Schr\"odinger equation}, Invent. Math. {\bf 181} (2010) 39 \arXiv{arXiv:0808.1742} [math.AP].

\bibitem{Hani1} Z.~Hani, {\em Long-time instability and unbounded Sobolev orbits for some periodic nonlinear Schr\"odinger equations}, Arch. Rational Mech. Anal. {\bf 211} (2014) 929 \arXiv{1210.7509} [math.AP].

\bibitem{Guardia3} M.~Guardia and V.~Kaloshin, {\em Growth of Sobolev norms in the cubic defocusing nonlinear Schr\"odinger equation}, J. Eur. Math. Soc. (JEMS), {\bf 17} (2015) 71 \arXiv{1205.5188} [math.AP].

\bibitem{Guardia} M.~Guardia, {\em Growth of Sobolev norms in the cubic nonlinear Schr\"odinger equation with a convolution potential}, Commun. Math. Phys. {\bf 329} (2014) 405 \arXiv{1211.1267} [math.AP].

\bibitem{Hani2} Z.~Hani, B.~Pausader, N.~Tzvetkov and N.~Visciglia, {\em Modified scattering for the cubic Schr\"odinger equation on product spaces and applications}, Forum Math. Pi {\bf 3} (2015) e4 \arXiv{1311.2275} [math.AP].

\bibitem{Guardia2} M.~Guardia, E.~Haus, and M.~Procesi, {\em Growth of Sobolev norms for the defocusing analytic NLS on $\mathbb{T}^2$}, Adv. Math. {\bf 301} (2016) 615 \arXiv{1503.02468} [math.AP].

\bibitem{Guardia4} M.~Guardia, Z.~Hani, E.~Haus, A.~Maspero and M.~Procesi, {\em Strong nonlinear instability and growth of Sobolev norms near quasiperiodic finite-gap tori for the 2D cubic NLS equation} \arXiv{1810.03694} [math.AP].

\bibitem{Merle}
F.~Merle, P.~Raphael, I.~Rodnianski and J.~Szeftel, {\em On blow up for the energy super critical defocusing non linear Schr\"odinger equations}, \arXiv{1912.11005} [math.AP].

\bibitem{BR} P.~Bizo\'n and A.~Rostworowski,
{\em On weakly turbulent instability of anti-de Sitter space,}
Phys.\ Rev.\ Lett.\ {\bf 107} (2011) 031102
\arXiv{1104.3702} [gr-qc].

\bibitem{rev2} B.~Craps and O.~Evnin,
{\em AdS (in)stability: an analytic approach,}
Fortsch.\ Phys.\ {\bf 64} (2016) 336
\arXiv{1510.07836} [gr-qc].

\bibitem{GG} P.~G\'erard and S.~Grellier,
{\em The cubic Szeg\H o equation,} Ann. Scient. \'Ec. Norm. Sup. {\bf 43} (2010) 761
\arXiv{0906.4540} [math.CV];\\
{\em Effective integrable dynamics for a certain nonlinear wave equation,} Anal. PDE {\bf 5} (2012) 1139 \arXiv{1110.5719} [math.AP];\\
\emph{An explicit formula for the cubic Szeg\H o equation,} Trans. Amer. Math. Soc. {\bf 367} (2015) 2979 \arXiv{1304.2619} [math.AP];\\
\emph{The cubic \Sz{} equation and Hankel operators} (Soc.\ Math.\ France, 2017)\\ \arXiv{1508.06814} [math.AP].

\bibitem{GG2} P.~G\'erard and S.~Grellier, {\em A survey of the \Sz{} equation}, Sci. China Math. {\bf 62} (2019) 1087.

\bibitem{Pocovnicu} O.~Pocovnicu,  {\em Traveling waves for the cubic \Sz{} equation on the real line}, Anal. \& PDE {\bf 4} (2011) 379 \arXiv{1001.4037} [math.AP];\\
{\em Explicit formula for the solution of the cubic \Sz{} equation on the real line and applications}, Discr.
Cont. Dyn. Sys. {\bf 31} (2011) 607 \arXiv{1012.2943} [math.AP].

\bibitem{Xu} H.~Xu, {\em Large-time blowup for a perturbation of the cubic \Sz{} equation}, Anal. PDE {\bf 7} (2014) 717
\arXiv{1307.5284} [math.AP];\\
\emph{The cubic \Sz{} equation with a linear perturbation},
\arXiv{1508.01500} [math.AP].

\bibitem{Thirouin} J.~Thirouin,
{\em Optimal bounds for the growth of Sobolev norms of solutions of a quadratic \Sz{} equation}, Trans. Amer. Math. Soc. {\bf 371} (2019) 3673
\arXiv{1710.01512} [math.AP];\\
{\em About the quadratic \Sz{} hierarchy}, SIAM J. Math. Anal. {\bf 51} (2019) 1454
\arXiv{1804.01261} [math.AP].

\bibitem{GG3} P.~G\'erard and S.~Grellier, {\em On a damped \Sz{} equation (with an appendix in collaboration with Christian Klein)},	\arXiv{1912.10933} [math.AP].

\bibitem{Xu2} H.~Xu, {\em Unbounded Sobolev trajectories and modified scattering theory for a wave guide nonlinear Schr\"odinger equation},
Math. Z. {\bf 286} (2017) 443 \arXiv{1506.07350} [math.AP].

\bibitem{GHT} P. Germain, Z. Hani and L. Thomann,
\emph{On the continuous resonant equation for NLS: I. Deterministic analysis,} J. Math. Pur. App. {\bf 105} (2016) 131
\arXiv{1501.03760} [math.AP].

\bibitem{GT} P.~Germain and L.~Thomann,  \emph{On the high frequency limit of the LLL equation}, Quart. Appl. Math. {\bf 74} (2016) 633 \arXiv{1509.09080} [math.AP].

\bibitem{BMP}A.~F.~Biasi, J.~Mas and A.~Paredes,
\emph{Delayed collapses of BECs in relation to AdS gravity,}
Phys.\ Rev.\ E {\bf 95} (2017) 032216 \arXiv{1610.04866} [nlin.PS].

\bibitem{BBCE}
A.~Biasi, P.~Bizo\'n, B.~Craps and O.~Evnin,
{\em Exact lowest-Landau-level solutions for vortex precession in Bose-Einstein condensates,}
Phys.\ Rev.\ A {\bf 96} (2017) 053615
\arXiv{1705.00867} [cond-mat.quant-gas];\\
{\em Two infinite families of resonant solutions for the Gross-Pitaevskii equation,} Phys.\ Rev.\ E {\bf 98} (2018) 032222
\arXiv{1805.01775} [cond-mat.quant-gas].

\bibitem{GGT}
P.~G\'erard, P.~Germain and L.~Thomann,
{\em On the cubic lowest Landau level equation,}  Arch. Rational Mech. Anal. {\bf 231} (2019) 1073
\arXiv{1709.04276} [math.AP].

\bibitem{FPU}
V.~Balasubramanian, A.~Buchel, S.~R.~Green, L.~Lehner and S.~L.~Liebling,
{\em Holographic thermalization, stability of anti-de Sitter space, and the Fermi-Pasta-Ulam paradox,}
Phys.\ Rev.\ Lett.\  {\bf 113} (2014) 071601
\arXiv{1403.6471} [hep-th].

\bibitem{CEV} B.~Craps, O.~Evnin and J.~Vanhoof,
\emph{Renormalization group, secular term resummation and AdS (in)stability,}
JHEP {\bf 1410} (2014) 48
\arXiv{1407.6273} [gr-qc];\\
\emph{Renormalization, averaging, conservation laws and AdS (in)stability,}
JHEP {\bf 1501} (2015) 108
\arXiv{1412.3249} [gr-qc].

\bibitem{BMR} P.~Bizo\'n, M.~Maliborski and A.~Rostworowski, \emph{Resonant dynamics and the instability of anti-de Sitter spacetime,} Phys.\ Rev.\ Lett. {\bf 115} (2015) 081103
\arXiv{1506.03519} [gr-qc].

\bibitem{CF}  P.~Bizo\'n, B.~Craps, O.~Evnin, D.~Hunik, V.~Luyten and M.~Maliborski,
\emph{Conformal flow on $S^3$ and weak field integrability in AdS$_4$,}  Comm.\ Math.\ Phys. {\bf 353} (2017) 1179 \arXiv{1608.07227} [math.AP].

\bibitem{BEL}
B.~Craps, O.~Evnin and V.~Luyten,
{\em Maximally rotating waves in AdS and on spheres,} JHEP {\bf 1709} (2017) 059
\arXiv{1707.08501} [hep-th].

\bibitem{BCE} A.~Biasi, B.~Craps and O.~Evnin,
{\em Energy returns in global AdS$_4$}, Phys.\ Rev.\ D {\bf 100} (2019)  024008
\arXiv{1810.04753} [hep-th].

\bibitem{BEF}
P.~Bizo\'n, O.~Evnin and F.~Ficek,
{\em A nonrelativistic limit for AdS perturbations},
JHEP {\bf 1812} (2018) 113
\arXiv{1810.10574} [gr-qc].


\bibitem{BBE}
A.~Biasi, P.~Bizo\'n and O.~Evnin,
{\em Solvable cubic resonant systems},
Comm.\ Math.\ Phys.\  {\bf 369} (2019) 433
\arXiv{1805.03634} [nlin.SI];\\
{\em Complex plane representations and stationary states in cubic and quintic resonant systems}, J.\ Phys.\ A {\bf 52} (2019)  435201
\arXiv{1904.09575} [math-ph].

\bibitem{breathing}
O.~Evnin,
{\em Breathing modes, quartic nonlinearities and effective resonant systems,} SIGMA \textbf{16} (2020) 034
\arXiv{1912.07952} [math-ph].

\bibitem{murdock}  J.~A.~Murdock, \emph{Perturbations: theory and methods}  (SIAM, 1987).

\bibitem{KM} S.~Kuksin and A.~Maiocchi, {\em The effective equation method}, in {\em New approaches to nonlinear waves} (Springer,  2016) \arXiv{1501.04175} [math-ph].

\bibitem{EP}
O.~Evnin and W.~Piensuk,
{\em Quantum resonant systems, integrable and chaotic,}
J.\ Phys.\ A {\bf 52} (2019) 025102
\arXiv{1808.09173} [math-ph].

\bibitem{QCh}F.~Haake, S.~Gnutzmann and M.~Ku\'s, {\em Quantum signatures of chaos} (Springer, 2018).

\end{thebibliography}
\end{document}